\documentclass{amcjoucc}

\usepackage{hyperref}
\usepackage[capitalize]{cleveref}
\usepackage{graphicx,amssymb}

\numberwithin{equation}{section}

\newtheorem{lem}[equation]{Lemma}
\newtheorem{prop}[equation]{Proposition}
\newtheorem{thm}[equation]{Theorem}
\newtheorem{cor}[equation]{Corollary}

\Crefname{prop}{Proposition}{Propositions}
\Crefname{cor}{Corollary}{Corollaries}

\theoremstyle{definition}
\newtheorem{notation}[equation]{Notation}
\newtheorem{eg}[equation]{Example}
\newtheorem{defn}[equation]{Definition}
\newtheorem{defns}[equation]{Definitions}
\newtheorem{rem}[equation]{Remark}
\newtheorem{rems}[equation]{Remarks}
\newtheorem*{ack}{Acknowledgments}

\crefformat{rems}{Remark~#2#1#3}

\crefmultiformat{eg}{Examples~#2#1#3}{ and~#2#1#3}{, #2#1#3}{, and~#2#1#3}

 \newcounter{case}
 \newenvironment{case}[1][\unskip]{\refstepcounter{case}
 \medskip \noindent \textbf{Case \thecase.}\em\ #1\ }{\unskip\upshape}
 \renewcommand{\thecase}{\arabic{case}}
 
  \newenvironment{case*}[1][\unskip]{\par \medbreak \em#1}{\unskip\upshape}

\crefformat{case}{Case~#2#1#3}
\Crefname{case}{Case}{Cases}

\crefformat{figure}{Figure~#2#1#3}

\newcommand{\pref}[1]{\textup(\ref{#1}\textup)}
\newcommand{\csee}[1]{\textup(see \cref{#1}\textup)}
\newcommand{\ccf}[1]{(cf.\ \cref{#1})}

\newcommand{\fullcref}[2]{\cref{#1}\pref{#1-#2}}
\newcommand{\fullcsee}[2]{(see \fullcref{#1}{#2})}
\newcommand{\fullref}[2]{\ref{#1}\pref{#1-#2}}

\DeclareMathOperator{\Aut}{Aut}
\DeclareMathOperator{\Cay}{Cay}
\newcommand{\Cayd}{\mathop{\overrightarrow\Cay}}
\DeclareMathOperator{\PSL}{PSL}

\usepackage[mathscr]{euscript}
\newcommand{\A}{\mathscr{A}}
\newcommand{\C}{\mathscr{C}}
\renewcommand{\L}{\mathscr{L}}
\newcommand{\N}{\mathscr{N}}

\newcommand{\iso}{\cong}
\newcommand{\normal}{\triangleleft}
\newcommand{\boxprod}{\mathbin{\raise1pt\hbox{\tiny$\square$}}}
\newcommand{\CHAR}{\mathrel{\raise1pt\hbox{$\blacktriangleleft$}}}
\def\edge{\mathbin{\hbox{\vrule height 3 pt depth -2.25pt width 10pt}}}

\renewcommand{\setminus}{\smallsetminus}
\newcommand{\ZZ}{\mathbb{Z}}

\newbox\circlebox
\setbox\circlebox\hbox{\large$\bigcirc$}
\newcommand{\circleit}[1]{%
	\hbox to \wd\circlebox{\hbox to 0pt{\copy\circlebox\hss}\hss$#1$\hss}}
\newcommand{\circlei}{\circleit{i}}

\newcommand{\omitit}[1]{}

\makeatletter
\newcommand{\noprelistbreak}{\@nobreaktrue\nopagebreak\smallskip} 
\makeatother

\newcommand{\MR}[1]{\href{http://www.ams.org/mathscinet-getitem?mr=#1}{MR\,#1}}

\makeatletter
\newcommand{\twoauthordatanofn}[5]{%
\advance\@authorcount by 2%
\xdef\@emails{\@emails\ifnum\@authorcount>2 , \fi {#4} (#1)}%
\xdef\@emails{\@emails, {#5} (#2)}%
\xdef\@authors{\@authors\ifnum\@authorcount>2 , \fi {#1}}
\xdef\@authors{\@authors, {#2}}
\begin{center}%
{\large #1, #2}\\[1mm]
{\itshape #3}
\end{center}}
\makeatother

\makeatletter
\renewenvironment{frontmatter}
{\thispagestyle{empty}}
{\vskip 20pt%
\blfootnote{\raggedright \ifnum\@authorcount=1\textit{E-mail address:}\else\textit{E-mail addresses:}\fi ~\@emails}%
\if@proofsline\global\linenumbers\fi%
}
\def\@oddrunninghead{On colour-preserving automorphisms of Cayley graphs}
\def\@evenrunninghead{A.\,Hujdurovi\'c, K.\,Kutnar, D.\,W.\,Morris, J.\,Morris}
\renewenvironment{abstract}
{
\begin{center}\small \today \end{center}\vskip 1mm
\hrule height 0.25pt
\vskip 5pt
\noindent \textbf{Abstract}
\vskip 5pt
}
{
\vskip 5pt
\noindent \textit{\small Keywords:~\@keywords}

\vskip 3pt
\noindent \textit{\small Math.\ Subj.\ Class.: \@msc}
\vskip 5pt
\hrule height 0.25pt
} 
\makeatother

\newcommand{\SeeAppendix}[1]{%
\marginparwidth=40pt
\marginparsep=10pt
\reversemarginpar
\marginpar{\vskip-2.5\baselineskip \it \raggedright \small \sf 
\hyperlink{AppendixForRef}{See \cref{#1} for some additional details.}}}

\begin{document}

\begin{frontmatter}

\titledata{On colour-preserving automorphisms \\ of Cayley graphs}{}

\twoauthordatanofn{Ademir Hujdurovi\'c}{Klavdija Kutnar}%
{University of Primorska, FAMNIT, Glagolja\v ska~8, 6000 Koper, Slovenia}%
{ademir.hujdurovic@upr.si}{klavdija.kutnar@upr.si}%

\twoauthordatanofn{Dave Witte Morris}{Joy Morris}%
{Department of Mathematics and Computer Science, University of Lethbridge, 
\\ Lethbridge, Alberta, T1K~3M4, Canada}%
{dave.morris@uleth.ca}{joy.morris@uleth.ca}%


\begin{abstract}
We study the automorphisms of a Cayley graph that preserve its natural edge-colouring. More precisely, we are interested in groups~$G$, such that every such automorphism of every connected Cayley graph on~$G$ has a very simple form: the composition of a left-translation and a group automorphism. We find classes of groups that have this property, and we determine the orders of all groups that do not have this property. We also have analogous results for automorphisms that permute the colours, rather than preserving them.
\end{abstract}

\end{frontmatter}

\tableofcontents

\section{Introduction}

\begin{defns} 
Let $S$ be a subset of a group~$G$, such that $S = S^{-1}$. (All groups and all graphs in this paper are finite.)
	\begin{itemize}
	\item The \emph{Cayley graph} of~$G$, with respect to~$S$, is the graph $\Cay(G;S)$ whose vertices are the elements of~$G$, and with an edge $x \edge xs$, for each $x \in G$ and $s \in S$. 	
	\item $\Cay(G;S)$ has a natural edge-colouring. Namely, each edge of the form $x \edge xs$ is coloured with the set $\{s,s^{-1}\}$. (In order to make the colouring well-defined, it is necessary to include $s^{-1}$, because $x \edge xs$ is the same as the edge $xs \edge x$, which is of the form $y \edge ys^{-1}$, with $y = xs$.)
	\end{itemize}
Note that $\Cay(G;S)$ is connected if and only if $S$ generates~$G$.
Also note that a permutation~$\varphi$ of~$G$ is a colour-preserving automorphism of $\Cay(G;S)$ if and only if we have $\varphi(xs) \in \bigl\{ \varphi(x)\, s^{\pm1} \bigr\}$, for each $x \in G$ and $s \in S$.
\end{defns}

For each $g \in G$, the left translation $x \mapsto gx$ is a colour-preserving automorphism of $\Cay(G;S)$. In addition, if $\alpha$ is an automorphism of the group~$G$, such that $\alpha(s) \in \{s^{\pm1}\}$ for all $s \in S$, then $\alpha$~is also a colour-preserving automorphism of $\Cay(G;S)$.
We will see that, in many cases, every colour-preserving automorphism of $\Cay(G;S)$ is obtained by composing examples of of these two obvious types.

\begin{defns} 
Let $G$ be a group.
\noprelistbreak
	\begin{enumerate}
	\item A function $\varphi \colon G \to G$ is said to be \emph{affine} if it is the  composition of an automorphism of~$G$ with left translation by an element of~$G$. This means $\varphi(x) = \alpha(gx)$, for some $\alpha \in \Aut G$ and $g \in G$.
	\item A Cayley graph $\Cay(G;S)$ is \emph{CCA} if all of its colour-preserving automorphisms are affine functions on~$G$.
(CCA is an abbreviation for the \emph{Cayley Colour Automorphism} property.)
	\item We say that $G$ is \emph{CCA} if every connected Cayley graph on~$G$ is CCA.
	\end{enumerate}
\end{defns}

Here are some of our main results:

\begin{thm} \label{CCASummary} \ 
\noprelistbreak
	\begin{enumerate}
	\item There is a non-CCA group of order~$n$ if and only if $n \ge 8$ and $n$~is divisible by either $4$, $21$, or a number of the form $p^q \cdot q$, where $p$ and~$q$ are prime\/ \csee{OrderOfNonCCA,OrderOfNonCCADivisiblePNotOdd}.
	\item \label{CCASummary-abel}
	An abelian group is not CCA if and only if it has a direct factor that is isomorphic to either $\ZZ_4 \times \ZZ_2$ or a group of the form $\ZZ_{2^k} \times \ZZ_2 \times \ZZ_2$, with $k \ge 2$ \csee{abelian}.
	\item Every dihedral group is CCA \csee{DihedralCCA}.
	\item No generalized dicyclic group or semidihedral group is CCA, except $\ZZ_4$ \csee{4NotCCA}.
	\item \label{CCASummary-odd}
	Every non-CCA group of odd order has a section that is isomorphic to either the nonabelian group of order\/~$21$ or a certain generalization of a wreath product\/ \textup(called a semi-wreathed product\/\textup) \csee{OddNotCCAHas}.
	\item If $G \times H$ is CCA, then $G$ and~$H$ are both CCA\/ \csee{NotCCAxAny}. The converse is not always true\/ \textup(for example, $\ZZ_4 \times \ZZ_2$ is not CCA\/\textup), but it does hold if\/ $\gcd \bigl( |G|, |H| \bigr) = 1$ \csee{GxHPrime}.
	\end{enumerate}
\end{thm}


We also consider automorphisms of $\Cay(G;S)$ that permute the colours, rather than preserving them:

\begin{defns} \ 
\noprelistbreak
	\begin{itemize}
	\item An automorphism $\alpha$ of a Cayley graph $\Cay(G;S)$ is \emph{colour-permuting} if it respects the colour classes; that is, if two edges have the same colour, then their images under~$\alpha$ must also have the same colour. This means there is a permutation $\pi$ of~$S$, such that $\alpha(gs) \in \{ \alpha(g) \, \pi(s)^{\pm1} \}$ for all $g \in G$ and $s \in S$
	 (and $\pi(s^{-1}) = \pi(s)^{-1}$).
	\item A Cayley graph $\Cay(G;S)$ is \emph{strongly CCA} if all of its colour-permuting automorphisms are affine functions on~$G$.
	\item We say that $G$ is \emph{strongly CCA} if every connected Cayley graph on~$G$ is strongly CCA.
	 \end{itemize}
\end{defns}

Note that every strongly CCA group is CCA, since colour-preserving automorphisms are colour-permuting (with $\pi$ being the identity map on~$S$). The converse is not true. For example, every dihedral group is CCA (as was mentioned above), but it is not strongly CCA if its order is of the form $8k + 4$ \csee{DihIff}. However, the converse does hold for at least two natural families of groups:

\begin{thm} \label{CCA<>strong} \ 
\noprelistbreak
A CCA group is strongly CCA if either:
	\begin{enumerate}
	\item \label{CCA<>strong-abel}
	it is abelian \csee{abelian},
	or
	\item it has odd order \csee{OddCCA->strong}.
	\end{enumerate}
\end{thm}

\begin{rems} \label{CCARems} \ 
\noprelistbreak
	\begin{enumerate}
	\item It follows from Theorems \fullref{CCASummary}{abel} and \fullref{CCA<>strong}{abel} 
that every cyclic group is strongly CCA. This is also a consequence of the main theorem of \cite{Morris-AutCircPart}. 

\item \label{CCARems-small}
Groups of even order seem far more likely to fail to be strongly CCA than groups of odd order. For example, of the $28$ groups of order less than~$32$ that are not strongly CCA, only one has odd order \csee{SmallGrpSect}. In fact, there are only three groups of odd order less than 100 that are not strongly CCA: the non-abelian group $G_{21}$ of order~$21$, the group $G_{21} \times \ZZ_3$ of order $63$, and the wreath product $\ZZ_3 \wr \ZZ_3$, which has order~$81$ \csee{OddLess100}.

\item \label{CCARems-normal}
If the subgroup consisting of all left-translations is normal in the automorphism group of the Cayley graph $\Cay(G;S)$, then $\Cay(G;S)$ is said to be \emph{normal} \cite{Xu-AutGrpsIsos}. It is not difficult to see that every normal Cayley graph is strongly CCA \ccf{Aff=N(G)}, and that every automorphism of a normal Cayley graph is colour-permuting.

\item \label{CCARems-Directed}
The notion of (strongly) CCA generalizes in a natural way to the setting of Cayley digraphs $\Cayd(G;S)$, by putting the colour~$s$ on each directed edge of the form $x \rightarrow xs$. (There is no need to include $s^{-1}$ in the colour.) However, it is very easy to see that if $\Cayd(G;S)$ is connected, then every colour-preserving automorphism of $\Cayd(G;S)$ is left-translation by some element of~$G$ \cite[Thm.~4-8, p.~25]{White-GrfsGrpsSurfs}, and that every colour-permuting automorphism is affine \cite[Lem.~2.1]{FiolFiolYebra-LineDigraph}.
Therefore, both notions are completely trivial in the directed setting. However, there has been some interest in determining when every automorphism of $\Cayd(G;S)$ is colour-permuting \cite{AlbertEtal-ColPermAuts,Dobson-SomeNonNormal} (in which case, the Cayley digraph is normal, in the sense of~\pref{CCARems-normal}).

\end{enumerate}
\end{rems}

\begin{ack}
We thank an anonymous referee for numerous helpful comments that improved the exposition.

D.\,W.\,M.\ and J.\,M.\ thank the Faculty of Mathematics, Natural Sciences and Information Technologies of the University of Primorska (Slovenia) for its hospitality during the visit that gave rise to this research project.

The work of A.\,H.\ was partially supported by research program P1-0285 from the Slovenian Research Agency. 
The work of K.\,K.\ was partially supported by research program P1-0285 and research projects N1-0011, J1-6743, and J1-6720 from the Slovenian Research Agency.
The work of J.\,M.\ was partially supported by a research grant from the Natural Sciences and Engineering Research Council of Canada.
\end{ack}

\section{Examples of non-CCA groups} \label{EgSect}

\begin{rem} \label{CCAFixId}
Since automorphisms are the only affine functions that fix the identity element~$e$ (and left-translations are colour-preserving automorphisms of any Cayley graph), it is easy to see that if $\Cay(G;S)$ is CCA, then every colour-preserving automorphism that fixes the identity is an automorphism of the group~$G$. More precisely:
	\begin{itemize}
	\item[] A Cayley graph $\Cay(G;S)$ is CCA if and only if, for every colour-preserving automorphism~$\varphi$ of $\Cay(G;S)$, such that $\varphi(e) = e$, we have $\varphi \in \Aut G$.
	\end{itemize}
The same is true with ``strongly CCA'' in the place of ``CCA\rlap,'' if ``colour-preserving'' is replaced with ``colour-permuting\rlap.''
This is reminiscent of the CI (Cayley Isomorphism) property \cite{Li-IsoCaySurvey}, and this similarity motivated our choice of terminology.
\end{rem}

We thank Gabriel Verret for pointing out that the quaternion group~$Q_8$ is not CCA. In fact, two different groups of order~$8$ are not CCA:

\begin{eg}[G.\,Verret] \label{Z4xZ2andQ8}
$\ZZ_4 \times \ZZ_2$ and $Q_8$ are not CCA.
\end{eg}

\begin{proof}
($Q_8$) Let $\Gamma=\Cay(Q_8; \{\pm i, \pm j\})$. This is the complete bipartite graph $K_{4,4}$. (See \cref{fig:notCCA} with the labels that are inside the vertices.) Let $\varphi$ be the graph automorphism that interchanges the vertices $k$ and $-k$ while fixing every other vertex. This is clearly not an automorphism of~$G$ since $i$ and $j$ are fixed by $\varphi$ and generate~$G$, but $\varphi \neq 1$. It is, however, a colour-preserving automorphism of~$\Gamma$. 

\begin{figure}[t]
\centerline{\includegraphics{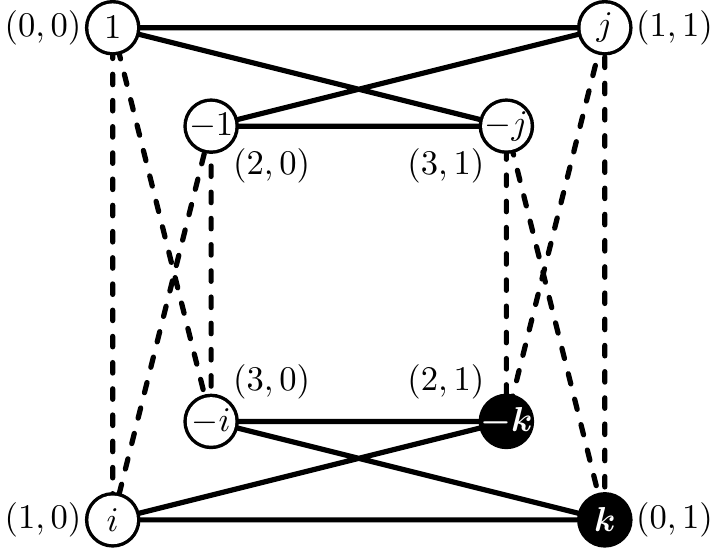}}
\caption{Interchanging the two black vertices while fixing all of the white vertices is a colour-preserving graph automorphism that fixes the identity vertex but is not a group automorphism.}\label{fig:notCCA}
\end{figure}
%
%
%
%
%
%
%
%

($\ZZ_4 \times \ZZ_2$)
Let $\Gamma=\Cay \bigl( \ZZ_4\times\ZZ_2; \{\pm(1,0),\pm(1,1)\} \bigr)$. This is again the complete bipartite graph $K_{4,4}$. (See \cref{fig:notCCA} with the labels that are outside the vertices.) Let $\varphi$ be the graph automorphism that interchanges the vertices $(0,1)$ and $(2,1)$ while fixing all of the other vertices. This is clearly not an automorphism of $G$ since $(1,0)$ and $(1,1)$ are fixed by $\varphi$ and generate $G$, but $\varphi \neq 1$. It is, however, a colour-preserving automorphism of~$\Gamma$.
\end{proof}

Both of the groups in \cref{Z4xZ2andQ8} are generalized dicyclic \ccf{DicyclicDefn}:
	\begin{itemize}
	\item $Q_8$ is the generalized dicyclic group over~$\ZZ_4$,
	and
	\item $\ZZ_4 \times \ZZ_2$ is the generalized dicyclic group over~$\ZZ_2 \times \ZZ_2$.
	\end{itemize}
More generally, we will see in \fullcref{4NotCCA}{dicyclic} below that no generalized dicyclic group is CCA, except~$\ZZ_4$.

We will see in \cref{OddNotCCAHas} that the following example is the smallest group of odd order that is not CCA.

\begin{eg} \label{21NotCCA}
The nonabelian group of order $21$ is not CCA.
\end{eg}

\begin{proof}
Let $G = \langle\, a,x \mid a^3 = e, \ a^{-1} x a = x^2 \,\rangle$. 
(Since $x = e^{-1} x e = a^{-3} x a^3 = x^8$, the relations imply $x^7 = e$, so $G$ has order~$21$.) By letting $b = ax$, we see that $G$ also has the presentation
	$$ G = \langle\, a,b \mid a^3 = e, \ (ab^{-1})^2 = b^{-1} a \,\rangle .$$
As illustrated in \cref{21NotCCAFig}, every element of~$G$ can be written uniquely in the form
	$$ \text{$a^i b^j a^k$, \ where \ $i,j,k \in \{0,\pm1\}$ \ and \ $j = 0 \Rightarrow k = 0$} .$$
Define 
	$$ \varphi(a^i b^j a^k) = \begin{cases}
	\hfil b^j a^{-k} & \text{if $i = 0$} , \\
	\hfil a b^{-j} a^k & \text{if $i = 1$} , \\
	a^{-1} b^{-j} a^{-k} & \text{if $i = -1$}
	. \end{cases} $$
Then $\varphi$ is a colour-preserving automorphism of $\Cay \bigl( G; \{a^{\pm1}, b^{\pm1} \} \bigr)$ \csee{21NotCCAFig}. However, $\varphi$ is not affine, since it fixes~$e$, but is not an automorphism of~$G$ (because $\varphi(ab) = a b^{-1} \neq a b = \varphi(a) \, \varphi(b)$).
\begin{figure}[t]
\centerline{\includegraphics{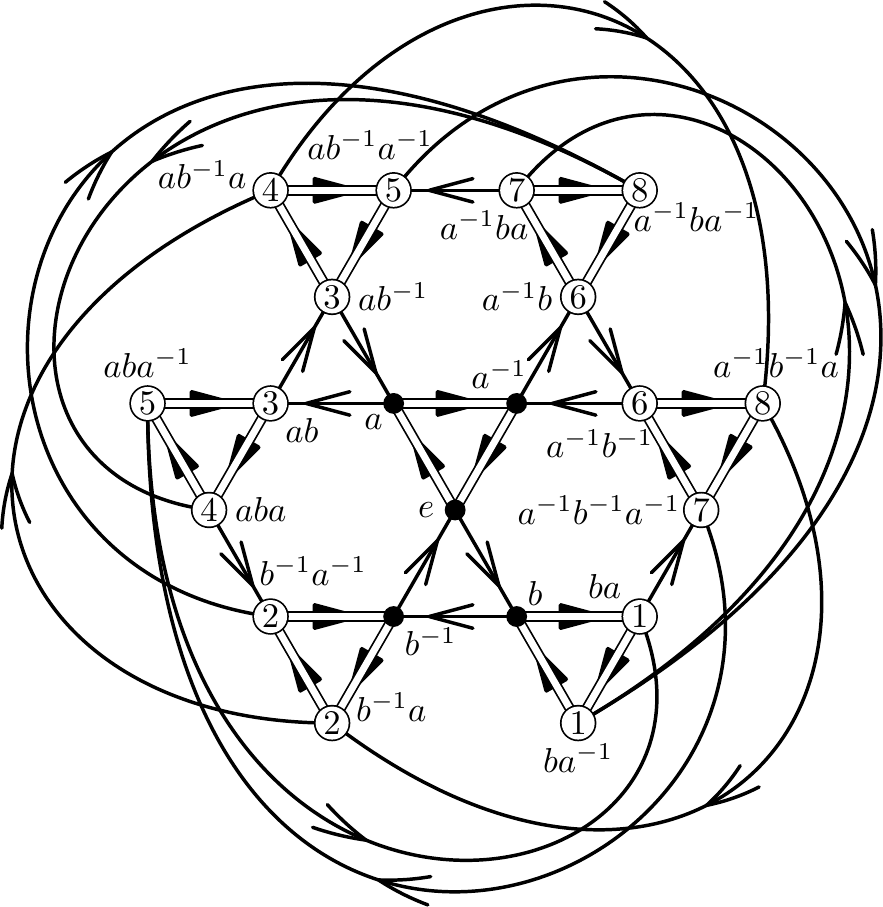}}
\caption{The colour-preserving automorphism~$\varphi$ fixes every black vertex, but interchanges the two vertices labeled~\circlei, for $1 \le i \le 8$. Since the neighbours of both copies of~\circlei\ have the same labels (for example, the vertices labeled~\circleit7 are connected by a black edge to \circleit1 and~\circleit5, and by a white edge to \circleit6 and~\circleit8), we see that $\varphi$ is indeed a colour-preserving automorphism of the graph (if the orientations of the edges are ignored).}
\label{21NotCCAFig}
\end{figure}
\end{proof}

See \cref{SemiNotCCA} for a generalization of the following example.

\begin{eg} \label{WreathProdNotCCA}
The wreath product $\ZZ_m \wr \ZZ_n$ is not CCA whenever $m \ge 3$ and $n \ge 2$. 
\end{eg}

\begin{proof}
This group is a semidirect product
        $$(\ZZ_m \times \ZZ_m \times \cdots \times \ZZ_m) \rtimes \ZZ_n. $$
For the generators $a = \bigl( (1,0,0,\ldots,0), 0 \bigr)$ and $b = \bigl( (0,0,\ldots,0), 1 \bigr)$, the map
        $$\bigl( (x_1,x_2,x_3,\ldots,x_n) , y \bigr)   \mapsto  \bigl( (-x_1, x_2, x_3, \ldots, x_n) , y \bigr) $$
(negate a single factor of the abelian normal subgroup) is a colour-preserving automorphism of $\Cay \bigl( \ZZ_m \wr \ZZ_n; \{a^{\pm1},b^{\pm1}\} \bigr)$ that fixes the identity element but is not a group automorphism.
\end{proof}

The following construction provides many additional examples of non-CCA groups by generalizing the idea of \cref{Z4xZ2andQ8}.

\begin{prop}\label{notCCA-structure}
Suppose there is a generating set~$S$ of~$G$, an element $\tau$ of~$G$, and a subset~$T$ of~$S$, such that:
	\begin{itemize}
	\item $S = S^{-1}$,
	\item $\tau$ is an element of order~$2$,
	\item each element of $S$ is either centralized or inverted by~$\tau$,
	\item $t^2 = \tau$ for all $t \in T$,
	\item the subgroup $\langle (S \setminus T) \cup \{\tau\} \rangle$ is not all of~$G$,
	and
	\item either $\bigl| G : \langle (S \setminus T) \cup \{\tau\} \rangle \bigr| > 2$ or $\tau$ is not in the centre of~$G$.
	\end{itemize}
Then $G$ is not CCA.
\end{prop}

\begin{proof}
For convenience, let $H = \langle (S \setminus T) \cup \{\tau\} \rangle$. Since $\langle S \rangle = G$, but, by assumption, $H \neq G$, there exists some $x \in T \setminus H$. Define 
	$$\varphi(g) = \begin{cases}
	g \tau & \text{if $g \in xH$} , \\
	g & \text{otherwise}
	. \end{cases}$$
It is obvious that $\varphi$ fixes~$e$, since $e \notin xH$. 

We claim that $\varphi$ is is not an automorphism of~$G$. If $| G : H | > 2$, this follows from the fact that a nonidentity automorphism cannot fix more than half of the elements of~$G$. Thus, we may assume $| G : H | = 2$. Then, by assumption, there is some element~$h$ of~$G$ that does not commute with~$\tau$. Since $\tau$ commutes with every element of~$T$ (because $\tau = t^2$), we see that we may assume $h \in H$. If $\varphi$ is an automorphism, then, since it is the identity on the normal subgroup~$H$ of~$G$, but $x^{-1} = x x^{-2} = x \tau \in x H$, we have:
	$$ x^{-1} h x
	= \varphi(x^{-1} h x) 
	= \varphi(x^{-1}) \cdot \varphi(h) \cdot \varphi(x) 
	= x^{-1} \tau \cdot h \cdot x \tau
	\neq x^{-1} h x \tau^2
	= x^{-1} h  x
	. $$
This is a contradiction.

Since each element of $S$ is either centralized or inverted by~$\tau$, we know that right-multiplication by~$\tau$ is a colour-preserving automorphism of $\Cay(G;S)$. Restricting to $xH$, this tells us that $\varphi$ preserves colours (and existence) of all edges of $\Cay(G;S)$ that have both endvertices in $xH$. 

Now consider an edge from $g$ to~$h$, where $g \in xH$ and $h \not\in xH$. There is some element $t \in T$ such that $gt=h$, and
there is an edge of the same colour from $\varphi(g) = g \tau$ to $g \tau t^{-1}$. 
Since $t^2 = \tau$ and $\tau^2 = e$, we have $t^{-1}=\tau t$. 
Hence, the edge is from $\varphi(g)$ to 
	$$g\tau t^{-1}=gt^2 t^{-1}=gt=h=\varphi(h) .$$
Thus $\varphi$ preserves the existence and colour of every edge from a vertex in $xH$ to a vertex outside of $xH$.
Since the only vertices moved by~$\varphi$ are in $xH$, this shows that $\varphi$ is a colour-preserving automorphism of $\Cay(G;S)$.
\end{proof}

\begin{defn} \label{DicyclicDefn}
Let $A$ be an abelian group of even order. Choose an involution $y$ of $A$. The corresponding \emph{generalized dicyclic group} is
	$$\mathrm{Dic}(A,y) = \langle\, x, A \mid  x^2=y, \ x^{-1} a x =a^{-1}, \ \forall a \in A\,\rangle .$$
\end{defn}

\begin{defn}
For $n \ge 1$, let
$$\mathrm{SemiD}_{16n}=\langle\, a, x \mid a^{8n}=x^2=e, \ xa=a^{4n-1}x\,\rangle.$$
This is a \emph{semidihedral} (or \emph{quasidihedral}) group. The term is usually used only when $n$ is a power of~$2$, but the construction is valid more generally.
\end{defn}

We have already seen in \cref{Z4xZ2andQ8} that $\ZZ_4 \times \ZZ_2$ and $Q_8$ are not CCA.
Here are a few additional examples that come from \cref{notCCA-structure}.

\begin{cor} \label{4NotCCA}
The following groups are not CCA:
	\begin{enumerate}

	\item \label{4NotCCA-BigAbel}
	$\ZZ_{2^k} \times \ZZ_2 \times \ZZ_2$, for any $k \ge 2$,

	\item \label{4NotCCA-dicyclic}
	every generalized dicyclic group except\/~$\ZZ_4$,
	and
	
	\item \label{4NotCCA-semidihedral}
	every semidihedral group.
	
	\end{enumerate}
\end{cor}

\begin{proof}
\pref{4NotCCA-BigAbel}
Apply \cref{notCCA-structure} with $\tau = (2^{k-1},0,0)$, $T = \{(2^{k-2},1,0),(2^{k-2},0,1) \}$, and $S = \{(1,0,0)\} \cup T$.

\pref{4NotCCA-dicyclic} For $G=\mathrm{Dic}(A,y) = \langle x, y, A \rangle$, apply \cref{notCCA-structure} with $\tau = y$ and $S = T = xA$.
(We have 
$\bigl| G : \langle (S \setminus T) \cup \{\tau\} \rangle \bigr| 
=  | G : \langle\tau \rangle | = |G|/ 2 > 2$, since $G \not\iso \ZZ_4$.)

\pref{4NotCCA-semidihedral}
For $G=\mathrm{SemiD}_{16n}=\langle a,x \rangle$, apply \cref{notCCA-structure} with $\tau = a^{4n}$,
$T = \{(ax)^{\pm 1}\}$, and $S = \{x\} \cup T$. 
(Note that $\bigl| G : \langle (S \setminus T) \cup \{\tau\} \rangle \bigr| 
=  | G : \langle x, \tau \rangle | = |G|/ 4 \ge 4$.)
\end{proof}

\section{Direct products and semidirect products} \label{DirProdSect}

\begin{prop} \label{NotCCAxAny}
If $G_1$ is not strongly CCA, and $G_2$ is any group, then $G_1 \times G_2$ is not strongly CCA.
Furthermore, the same is true with ``CCA'' in the place of ``strongly CCA\rlap.''
\end{prop}

\begin{proof}
Since $G_1$ is not strongly CCA, some connected Cayley graph $\Cay(G_1;S_1)$ on~$G_1$ has a colour-permuting automorphism~$\varphi_1$ that is not affine. Let $\pi$ be a permutation of~$S_1$, such that $\varphi_1(g_1 s) \in \{ \varphi_1(g_1) \, \pi(s)^{\pm1}\}$ for all $g_1 \in G_1$ and $s \in S_1$. (If $G_1$ is not CCA, then we may assume $\pi$ is the identity permutation.) 
Now, fix any connected Cayley graph $\Cay(G_2; S_2)$ on~$G_2$, and let 
	$$S = \bigl( S_1 \times \{e\} \bigr) \cup \bigl( \{e\} \times S_2 \bigr) ,$$
so $\Cay (G_1 \times G_2 ; S )$ is connected.  (It is isomorphic to the Cartesian product $\Cay(G_1; S_1) \mathbin{\Box} \Cay(G_2; S_2)$.)

Define a permutation $\varphi$ of $G_1 \times G_2$ by 
	$\varphi (x_1,x_2) = \bigl( \varphi_1(x_1),x_2 \bigr)$.
For all $(x_1,x_2) \in G_1 \times G_2$ and $s_i \in S_i$, we have 
	\begin{itemize}
	\item $\varphi \bigl( (x_1,x_2) \cdot (s_1,e) \bigr) = \bigl( \varphi_1( x_1 s_1) , x_2 \bigr) \in \bigl\{ \varphi \bigl( x_1, x_2 \bigr) \cdot \bigl( \pi(s_1) , e \bigr)^{\pm1} \bigr\}$,
	and
	\item $\varphi \bigl( (x_1,x_2) \cdot (e,s_2) \bigr) = \bigl( \varphi_1( x_1) , x_2 s_2 \bigr) = \varphi (x_1, x_2) \cdot ( e, s_2 )$.
	\end{itemize}
Therefore, $\varphi$ is a colour-permuting automorphism of $\Cay (G_1 \times G_2 ; S )$ (and it is colour-preserving if $\pi$ is the identity permutation of~$S_1$).

To complete the proof that $G_1 \times G_2$ is not strongly CCA (and is not CCA if $\pi$ is the identity permutation of~$S_1$), it suffices to show that $\varphi$ is not affine. We prove this by contradiction: suppose there exists an automorphism $\alpha$ of $G_1 \times G_2$ and $(g_1,g_2) \in G_1 \times G_2$, such that $\varphi(x_1,x_2) = \alpha(g_1 x_1, g_2 x_2)$, for all $(x_1,x_2) \in G_1 \times G_2$. By the definition of~$\varphi$, this implies $\varphi_1(x_1) = \alpha(g_1 x_1)$ for all $x_1 \in G_1$. Since $\varphi_1(x_1) \in G_1$ and $g_1 x_1$ is an arbitrary element of~$G_1$, we conclude that $\alpha(G_1) \subseteq G_1$, so the restriction of~$\alpha$ to~$G_1$ is an automorphism of~$G_1$. Hence, the equation $\varphi_1(x_1) = \alpha(g_1 x_1)$ implies that $\varphi_1$ is affine. This contradicts the choice of~$\varphi_1$.
\end{proof}

\Cref{NotCCAxAny} tells us that if $G_1 \times G_2$ is CCA, then $G_1$ and~$G_2$ must both be CCA. The converse is not true. (For example, $\ZZ_4$ and $\ZZ_2$ are both CCA, but \cref{Z4xZ2andQ8} tells us that the direct product $\ZZ_4 \times \ZZ_2$ is not CCA.) However, the converse is indeed true when the groups are of relatively prime order:

\begin{prop} \label{GxHPrime}
Assume $\gcd \bigl( |G_1|, |G_2| \bigr) = 1$. Then $G_1 \times G_2$ is CCA\/ \textup(or strongly CCA\/\textup) if and only if $G_1$ and~$G_2$ are both CCA\/ \textup(or strongly CCA, respectively\/\textup).
\end{prop}

\begin{proof}
($\Rightarrow$) \Cref{NotCCAxAny}.

\medbreak

($\Leftarrow$) Let 
	\begin{itemize}
	\item $G = G_1 \times G_2$,
	\item $S$ be a generating set of $G$,
	\item $\varphi$~be a colour-permuting automorphism of $\Cay(G; S)$ that fixes the identity element \csee{CCAFixId}, 
	\item $\pi_i \colon G_1 \times G_2 \to G_i$ be the natural projection,
	and
	\item $k$ be a multiple of $|G_2|$ that is $\equiv 1 \pmod{|G_1|}$, so $g^k = \pi_1(g)$ for all $g \in G$.
	\end{itemize}

Consider some $s \in S$, and let $t = \varphi(s)$, so $\varphi(xs^i) = \varphi(x) \, t^{\pm i}$ for all $x \in G$ and $i \in \ZZ$.
Then, for all $g \in G$, we have 
	\begin{align} \tag{$*$} \label{IsCCA}
	\varphi \bigl( g \, \pi_1(s) \bigr) 
	&= \varphi( g s^k )
	= \varphi(g) \, t^{\pm k}
	= \varphi(g) \cdot \pi_1(t)^{\pm1}
	.\end{align}
Since $\pi_1(S)$ generates~$G_1$, this implies there is a well-defined permutation~$\varphi_2$ of~$G_2$, such that
	$$ \text{$\varphi(G_1 \times \{g_2\}) = G_1 \times \{\varphi_2(g_2)\}$ for all $g_2 \in G_2$} . $$ 
By repeating the argument with the roles of $G_1$ and~$G_2$ interchanged, we conclude that there is a permutation $\varphi_1$ of~$G_1$, such that
	$$ \text{$\varphi(g_1,g_2) = \bigl( \varphi_1(g_1), \varphi_2(g_2) \bigr)$ for all $(g_1,g_2) \in G_1 \times G_2$} . $$

Now, \eqref{IsCCA} implies that $\varphi_1$ is a colour-permuting automorphism of $\Cay \bigl( G_1; \pi_1(S) \bigr)$. Similarly,  $\varphi_2$ is a colour-permuting automorphism of $\Cay \bigl( G_2; \pi_2(S) \bigr)$.
Since each $G_i$ is CCA, we conclude that $\varphi_i$ is an automorphism of~$G_i$. So $\varphi$ is an automorphism of $G_1 \times G_2$.
\end{proof}

The idea used in \cref{WreathProdNotCCA} yields the following result that generalizes the CCA part of \cref{NotCCAxAny}. 

\begin{prop} \label{SemiNotCCA}
Suppose $G = H \rtimes K$ is a semidirect product, and $\Cay(H;S_0)$ is a connected Cayley graph of~$H$, such that:
	\begin{itemize}
	\item $S_0$ is invariant under conjugation by every element of~$K$,
	and
	\item there is a colour-preserving automorphism $\varphi_0$ of $\Cay(H;S_0)$, such that either 
		\begin{itemize}
		\item $\varphi_0$ is not affine, 
		or 
		\item $\varphi_0(e) = e$, and there exist $s \in S_0$ and $k \in K$, such that $\varphi_0(k^{-1} s k) \neq k^{-1} \, \varphi_0(s) \, k$.
		\end{itemize}
	\end{itemize}
Then $G$ is not CCA. 
\end{prop}

\begin{proof}
Define $\varphi \colon G \to G$ by 
	$\varphi( h k ) =  \varphi_0(h) \, k $.
We claim that $\varphi$ is a colour-preserving automorphism of $\Cay(G; S_0 \cup K)$ that is not affine (so $G$ is not CCA, as desired).

For $h \in H$ and $k,k_1 \in K$, we have 
	$$ \varphi (hk \, k_1 )
	= \varphi_0( h) \, k k_1 
	= \varphi(hk) \, k_1 ,$$
so $\varphi$ preserves the colour of $K$-edges.
Now consider some $s \in S_0$ and let ${}^k\!s = k s k^{-1} \in S_0$. Then, since $\varphi_0$ is colour preserving, we have
	$$ \varphi (hk \, s )
	= \varphi( h \, {}^k\!s \, k)
	= \varphi_0( h \, {}^k\!s ) \, k 
	= \bigl( \varphi_0(h) \, ({}^k\!s)^{\pm1} \bigr) \, k
	= \varphi_0(h)  \, k s^{\pm1}
	= \varphi(hk)  \,s^{\pm1}
	, $$
so $\varphi$ also preserves the colour of $S_0$-edges.
Hence, $\varphi$ is colour-preserving.

Now, suppose $\varphi$ is affine. Then the restriction $\varphi_0$ of~$\varphi$ to~$H$ is also affine, so, by assumption, we must have $\varphi(e) = e$, so $\varphi$ is an automorphism of~$G$. Hence, for all $s \in S_0$ and $k \in K$, we have 
	$$ \varphi_0(k^{-1} s k) = \varphi(k^{-1} s k) = \varphi(k)^{-1} \, \varphi(s) \, \varphi(k) = k^{-1} \, \varphi(s) \, k = k^{-1} \, \varphi_0(s) \, k .$$
This contradicts the hypotheses of the \lcnamecref{SemiNotCCA}.
\end{proof}

\begin{rem}
\Cref{SemiNotCCA} can be generalized slightly: assume $G = HK$ and $H \normal G$ (but do not assume $H \cap K = \{e\}$, which would make $G$ a semidirect product). Then the above proof applies if we make the additional assumption that $\varphi_0(hk) = \varphi_0(h) \, k$ for all $h \in H$ and $k \in H \cap K$.
\end{rem}

\section{Abelian groups} \label{AbelSect}

The following result shows that all non-CCA abelian groups can be constructed from examples that we have already seen in \cref{Z4xZ2andQ8,4NotCCA} (and that CCA and strongly CCA are equivalent for abelian groups).

\begin{prop} \label{abelian}
For an abelian group~$G$, the following are equivalent:
	\begin{enumerate} 
	\item \label{abelian-factor}
	$G$ has a direct factor that is isomorphic to either $\ZZ_4 \times \ZZ_2$ or a group of the form $\ZZ_{2^k} \times \ZZ_2 \times \ZZ_2$, with $k \ge 3$.
	\item \label{abelian-notCCA}
	$G$ is not CCA.
	\item \label{abelian-notstrong}
	$G$ is not strongly CCA.
	\end{enumerate} 
\end{prop}

\begin{proof}
($\ref{abelian-factor} \Rightarrow \ref{abelian-notCCA}$) This is immediate from \cref{Z4xZ2andQ8,4NotCCA,NotCCAxAny}.

($\ref{abelian-notCCA} \Rightarrow \ref{abelian-notstrong}$) Obvious.

($\ref{abelian-notstrong} \Rightarrow \ref{abelian-factor}$) 
We prove the contrapositive. Assume $G$ does not have any direct summands of the form specified in \pref{abelian-factor}. Given a connected Cayley graph $\Cay(G;S)$ on~$G$,  and a colour-permuting automorphism~$\varphi$ of $\Cay(G;S)$, such that $\varphi(0) = 0$, we will show that $\varphi$ is an automorphism of~$G$.

From \cref{GxHPrime} (and the fact that every abelian group is the direct sum of its Sylow subgroups), we may assume $G$ is a $p$-group for some prime~$p$.
Then 
	$$G \iso \ZZ_{p^{k_1}} \times \ZZ_{p^{k_2}} \times \cdots \times \ZZ_{p^{k_m}} ,
	\text{\quad with $k_1 \ge k_2 \ge \cdots \ge k_m \ge 1$} .$$
Since $S$ is a generating set, it is easy to see that there is some $s_1 \in S$, such that $|s_1| = p^{k_1}$. Also, it is a basic fact about finite abelian groups that every cyclic subgroup of maximal order is a direct summand \cite[Lem.~1.3.3, p.~10]{Gorenstein-FiniteGrps}. Therefore, by induction on~$i$, we see that there exist $s_1,\ldots,s_m \in S$, such that if we let $G_i = \langle s_1,\ldots, s_i \rangle$, then 
	$$ \text{$G_i \iso  G_{i-1} \times \ZZ_{p^{k_i}}$
	\quad and \quad
	$G \iso G_i \times \ZZ_{p^{k_{i+1}}} \times \cdots \times \ZZ_{p^{k_m}}$,
	\quad for each~$i$}
	. $$
It is important to note that each element of~$G_i$ can be written uniquely in the form 
	\begin{align} \tag{$\dagger$} \label{UniqRep}
	\text{$g + rs_i$, with $g \in G_{i-1}$ and $-p^{k_i}/2 < r \le p^{k_i}/2$
	(and $r \in \ZZ$)}
	. \end{align}
For convenience, also let
	$$ \text{$t_i = \varphi(s_i)$ \quad and \quad $H_i = \langle t_1,\ldots,t_i \rangle$} .$$

We will show, by induction on~$i$, that $H_i$ is a direct factor of~$G$, and the restriction of $\varphi$ to $G_i$ is an isomorphism onto~$H_i$. (Note that this implies $G/G_i \iso G/H_i$, by the uniqueness of the decomposition of~$G$ as a direct sum of cyclic groups.) Taking $i = m$ yields the desired conclusion that $\varphi$ is an automorphism of~$G$.

The base case $i = 0$ is trivial.
For the induction step, write $G = G_{i-1} \times \overline{G}$, so  
	$$\overline{G} \iso G/G_{i-1} \iso \ZZ_{p^{k_i}} \times \ZZ_{p^{k_{i+1}}} \times \cdots \times \ZZ_{p^{k_m}} ,$$
and let $\overline{\phantom{x}} \colon G \to \overline{G}$ be the natural projection.
Then $\langle \overline{s_i} \rangle = \overline{G_i} \iso \ZZ_{p^{k_i}}$ is a direct summand of~$\overline{G}$. 
Since $\varphi$ is colour-permuting (and $H_{i-1}= \varphi(G_{i-1})$ is a subgroup), it is easy to see that the order of $t_i$ in $G/H_{i-1}$ is equal to~$p^{k_i}$ (the same as the the order  of~$s_i$ in $G/G_{i-1}$), and that $\varphi ( p^{k_i} s_i ) = p^{k_i} t_i$. This implies that if we define
	$$ \text{$\alpha \colon G_i \to H_i$ \ by \ $\alpha(g + rs_i) = \varphi(g) + r t_i$
	\quad for $g \in G_{i-1}$ and $r \in \ZZ$} ,$$
then $\alpha$ is a well-defined isomorphism. So we need only show that the restriction of~$\varphi$ to~$G_i$ is equal to~$\alpha$.

Suppose $\varphi|_{G_i} \neq \alpha$. (This will lead to a contradiction.) 
Since $\varphi$ is colour-permuting and, by definition, $\alpha$ agrees with~$\varphi$ on $G_{i-1}$, this implies there is some $g \in G_{i-1}$, such that $\varphi(g + s_i) \neq \alpha(g+ s_i)$. However, since $\varphi$~is colour-permuting, we know
	$$\varphi( g + s_i ) = \varphi(g) \pm \varphi(s_i) = \alpha(g) \pm t_i.$$
Since $\alpha(g+ s_i) = \alpha(g) + t_i $, the preceding two sentences imply 
	$$\varphi( g + s_i ) = \alpha(g) - t_i \in H_{i-1} - t_i .$$
Furthermore, since $\varphi$ is colour-permuting (and $\varphi(s_j) = t_j$), we know that it maps edges of colour $\{s_1^{\pm1}\}, \ldots,\{s_{i-1}^{\pm1}\}$ to edges of colour $\{t_1^{\pm1}\}, \ldots,\{t_{i-1}^{\pm1}\}$, so
	$$ \text{$\varphi( x + h ) \in \varphi(x) + H_{i-1}$ for all $x \in G$ and $h \in H_{i-1}$} .$$ Taking $x = s_i$ and $h = g$ yields
	$$ \varphi( g + s_i ) \in H_{i-1} + \varphi(s_i) = H_{i-1} + t_i .$$
This contradicts the uniqueness of~$r$ in the analogue of \eqref{UniqRep} for~$H_i$, unless $1 = p^{k_i}/2$. Hence, we must have $p^{k_i} = 2$ (so $\ZZ_2$ is a direct summand of~$G$), which means $p = 2$ and $k_i = 1$.

We have
	\begin{align*}
	 \varphi(g) + 2t_i
	&= \alpha(g + 2s_i) && \text{(definition of~$\alpha$)}
	\\&= \varphi(g + 2s_i) && \text{($g + 2s_i  = g + p^{k_i} s_i \in G_{i-1}$)}
	\\&= \varphi(g) - 2t_i && \text{($\varphi( g + s_i ) = \alpha(g) - t_i = \varphi(g) - t_i$)}
	, \end{align*}
so $4 t_i = 0$. Also note that, since 
	$$\varphi(g) + t_i = \alpha(g + s_i) \neq \varphi(g + s_i) = \varphi(g) - t_i ,$$
we must have $2t_i \neq 0$. So $|t_i| = 4$.

Since $\langle s_1,\ldots,s_{i-1} \rangle = G_{i-1}$, there must exist $g' \in G_{i-1}$, and $j < i$, such that 
	$$ \text{$\varphi(g' + s_i) = \alpha(g') + t_i$, 
	\  but \ 
	$\varphi(g' + s_j + s_i) = \alpha(g' + s_j) - t_i = \alpha(g') + t_j - t_i$}
	. $$
Since $\varphi$ is colour-permuting, we also have
	$$ \varphi(g' + s_j + s_i) = \varphi(g' + s_i) \pm t_j = \alpha(g') + t_i \pm t_j .$$
Hence, $t_j - t_i = t_i \pm t_j$, so $t_j \mp t_j = 2 t_i$. Since $2 t_i \neq 0$, we conclude that $2 t_j = 2t_i$; hence, $|t_j| = 4$.

Since $2^{k_j} = |H_j : H_{j-1}|$ is a divisor of~$|t_j|$, and $|t_j| = 4$, there are two possibilities for $k_j$:
	\begin{itemize}
	\item If $k_j = 2$, then $\ZZ_4 \times \ZZ_2 \iso \ZZ_{2^{k_j}} \times \ZZ_{2^{k_i}}$ is a direct summand of~$G$.
	
	\item If $k_j = 1$, then, since $|t_j| = 4$, there must be some $\ell < j$, such that $k_\ell \ge 2$. This implies that $\ZZ_{2^{k_\ell}} \times \ZZ_2 \times \ZZ_2 \iso \ZZ_{2^{k_\ell}} \times \ZZ_{2^{k_j}} \times \ZZ_{2^{k_i}}$ is a direct summand of~$G$.
	\end{itemize}
Each of these possibilities contradicts our assumption that there are no direct summands as specified in~\pref{abelian-factor} of the statement of the \lcnamecref{abelian}.
\end{proof}

\begin{cor} \label{OrderOfNotCCAAbel}
For $n \in \ZZ^+$, there is a non-CCA abelian group of order~$n$ if and only if $n$~is divisible by~$8$.
\end{cor}

\section{Generalized dihedral groups} \label{DihedralSect}

\begin{defn}
The \emph{generalized dihedral group} over an abelian group~$A$ is the group
	$$ \langle\, \sigma, A \mid \sigma^2 = e, \ \sigma a \sigma = a^{-1}, \ \forall a \in A \,\rangle .$$
\end{defn}

\begin{lem} \label{DihCCALem}
Suppose $D$ is the generalized dihedral group over an abelian group~$A$, and $\varphi$ is a colour-permuting automorphism of a connected Cayley graph\/ $\Cay(D;S)$, such that $\varphi(e) = e$. If $A$~is strongly CCA, and $\varphi(S \cap A) = S \cap A$, then $\varphi$~is an automorphism of~$D$.
\end{lem}

\begin{proof}
Label the elements of $S$ as $S=\{a_1,a_2, \ldots, a_k, \sigma_1, \sigma_2, \ldots, \sigma_t\}$, where $a_i \in A$ for $1 \le i \le k$, and $\sigma_i \not\in A$ for $1 \le i \le t$ (so each $\sigma_i$ is an involution whose action by conjugation inverts every element of~$A$). 
By assumption, $\{a_1,a_2, \ldots, a_k\}$ and $\{\sigma_1, \sigma_2, \ldots, \sigma_t\}$ are invariant under~$\varphi$. Thus, for each~$i$, we have
	\begin{itemize}
	\item $\varphi(a_i) = a'_i$ for some $a'_i \in \{a_1,a_2, \ldots, a_k\}$,
	and
	\item $\varphi(\sigma_i) = \sigma'_i$ for some $\sigma'_i \in \{\sigma_1, \sigma_2, \ldots, \sigma_t\}$.
	\end{itemize}

Notice that since $\sigma_1, \ldots, \sigma_t$ are involutions, each $\sigma_i$ is its own inverse. Therefore, 
whenever $\sigma$ is a word in $\sigma_1, \ldots, \sigma_t$ and $g \in D$, the fact that $\varphi$ is a colour-permuting automorphism means that 
$\varphi(g\sigma)=\varphi(g)\sigma'$, where $\sigma'$ is formed from $\sigma$ by replacing each instance of $\sigma_i$ in $\sigma$ by $\sigma'_i$. Therefore, if we let $\Sigma$ be the subgroup generated by $\{\sigma_1, \ldots, \sigma_t\}$, then $\varphi$ is a colour-preserving automorphism of the Cayley graph $\Cay(D; S \cup \Sigma)$. Hence, there is no harm in assuming that $S = S \cup \Sigma$, so $\Sigma \subseteq S$.

Since $\langle S \cap A \rangle$ is normal in~$D$ (in fact, every subgroup of~$A$ is normal, because every element of~$D$ either centralizes or inverts it), we have $D = \langle S \cap A \rangle \Sigma$. Therefore $A = \langle S \cap A \rangle (\Sigma \cap A) = \langle S \cap A \rangle$, so $\Cay( A ; S \cap A )$ is connected. Since $\varphi$ is colour-preserving, and $\varphi(S \cap A) = S \cap A$, this implies that $\varphi(A) = A$. So $\varphi$ is a colour-permuting automorphism of the connected Cayley graph $\Cay( A ; S \cap A )$. Since, by assumption, $A$ is strongly CCA, this implies that $\varphi|_A$ is an automorphism of~$A$. So $\varphi(a b^\epsilon) = \varphi(a) \, \varphi(b)^\epsilon$ for all $a,b \in A$ and $\epsilon \in \ZZ$.

Now we are ready to show that $\varphi$ is an automorphism of $D$. 
Let $g,h \in D$. Then we may write $g = a\sigma$ and $h = b \widetilde\sigma$, where $a,b \in A$ and $\sigma,\widetilde\sigma \in \{e,\sigma_1\}$. For convenience, let $\epsilon \in \{\pm1\}$, such that $\sigma c \sigma = c^\epsilon$ for all $c \in A$. Note that, since $\sigma_1' \in \{\sigma_1,\ldots,\sigma_t\}$, we know that $\sigma_1$ and~$\sigma_1'$ both invert~$A$, so we also have $\sigma' c \sigma' = c^\epsilon$. Then
	$$ \varphi(gh)
	= \varphi( a \sigma \cdot b \widetilde\sigma) 
	= \varphi(ab^\epsilon \cdot \sigma \widetilde\sigma)
	= \varphi(a) \, \varphi(b)^\epsilon \cdot \sigma' \widetilde\sigma'
	=  \varphi(a) \sigma' \cdot \varphi(b) \widetilde\sigma'
	= \varphi(g) \cdot \varphi(h) .$$
Since $g, h \in D$ are arbitrary, this proves that $\varphi$ is an automorphism of $D$.
\end{proof}

\begin{prop} \label{GenDihedralCCA}
The generalized dihedral group~$D$ over an abelian group~$A$ is CCA if and only if $A$~is CCA.
\end{prop}

\begin{proof}
($\Leftarrow$) Note that if $\varphi$ is any colour-preserving automorphism of a connected Cayley graph $\Cay(D;S)$ such that $\varphi(e) = e$, then $\varphi(S \cap A) = S \cap A$, since $A$~is closed under inverses. Furthermore, $A$ is strongly CCA, since it is assumed to be CCA and every CCA abelian group is strongly CCA \csee{abelian}. Therefore, \cref{DihCCALem} implies that $\varphi$ is a group automorphism. So $D$ is CCA.

($\Rightarrow$) Write $D = A \rtimes \langle \sigma \rangle$. Since $A$ is not CCA, there is a colour-preserving automorphism $\varphi_0$ of some connected Cayley graph $\Cay(A;S)$, such that $\varphi_0$ is not affine. Since $\sigma$ inverts every element of~$S$, it is easy to see that $\Cay \bigl( D; S\cup \{\sigma\} \bigr)$ is isomorphic to the Cartesian product $\Cay(A; S) \mathbin{\Box} P_2$. So the proof of \cref{NotCCAxAny} provides a colour-preserving automorphism~$\varphi$ of $\Cay \bigl( D; S\cup \{\sigma\} \bigr)$ whose restriction to~$A$ is~$\varphi_0$, which is not an affine map. Therefore, $\varphi$ is not affine.
\end{proof}

The following result is the special case where $A$~is cyclic (since \cref{abelian} implies that every cyclic group is CCA).

\begin{cor} \label{DihedralCCA}
Every dihedral group is CCA.
\end{cor}

\begin{lem} \label{LikeDihedralNotCCA}
If $T$ is a generating set of a group~$H$, and $\sigma$ is a nontrivial automorphism of~$H$, such that $\sigma(t) \in \{t^{\pm1}\}$ for every $t \in T$, then the group $G = (H \rtimes \langle \sigma \rangle) \times \ZZ_2$ is not strongly CCA.
\end{lem}

\begin{proof}
Let $G' = H \times \ZZ_2 \times \ZZ_2$ and define $\varphi \colon G \to G'$ by
	$ \varphi(h,\sigma^x,y) = (h,x,y)$ for $h \in H$ and $x,y \in \ZZ_2$.
Since $\sigma(t) \in \{t^{\pm1}\}$ for every~$t$, it is easy to verify that $\varphi$ is a colour-respecting isomorphism 
	\begin{align*}
	\text{from } &\Cay \bigl( G ; (H,e,0) \cup \{(e,\sigma,0), (e,0,1)\} \bigr) \\
	\text{to } & \Cay \bigl( G ; (H,0,0) \cup \{(e,1,0), (e,0,1)\} \bigr)
	. \end{align*}
Permuting the two $\ZZ_2$ factors of~$G'$ provides an automorphism of~$G'$ that preserves the generating set, and therefore corresponds to a colour-permuting automorphism of the two Cayley graphs. However, it is not an automorphism of~$G$, since it takes the central element $(e,e,1)$ to $(e,\sigma,0)$, which is not central (since the automorphism~$\sigma$~is nontrivial).
\end{proof}

\begin{prop} \label{DihIff}
The generalized dihedral group over an abelian group~$A$ is strongly CCA if and only if either $A$~does not have $\ZZ_2$ as a direct factor, or $A$ is an elementary abelian $2$-group\/ \textup(in which case, the generalized dihedral group is also an elementary abelian $2$-group\/\textup).
\end{prop}

\begin{proof}
($\Rightarrow$)
Suppose $A = A' \times \ZZ_2$, and $A'$ is not elementary abelian. 
Then the generalized dihedral group $A \rtimes \langle \sigma \rangle$ over~$A$ is isomorphic to $(A' \rtimes \langle \sigma \rangle) \times \ZZ_2$, so \cref{LikeDihedralNotCCA} tells us that it is not strongly CCA.

\medbreak

($\Leftarrow$)
Let $D = A \rtimes \langle \sigma \rangle$ be the generalized dihedral group over~$A$, and let $\varphi$ be a colour-permuting automorphism of a connected Cayley graph $\Cay(D;S)$, such that $\varphi(e) = e$. We may assume $A$~does not have $\ZZ_2$ as a direct factor (otherwise, the desired conclusion follows from the fact that every elementary abelian $2$-group is strongly CCA \csee{abelian}). From \cref{abelian}, we see that $A$ is strongly CCA. Hence, the desired conclusion will follow from \cref{DihCCALem} if we show that $\varphi(S \cap A) = S \cap A$.

Let $a \in S \cap A$. Since $\varphi$ is colour-permuting, we have $|\varphi(s)| = |s|$ for all $s \in S$. Also, we know that $|g| = 2$ for all $g \in D \setminus A$. Therefore, it is obvious that $\varphi(a) \in S \cap A$ if $|a| \neq 2$.

So we may assume $|a| = 2$. Since $A$~does not have $\ZZ_2$ as a direct factor, this implies that $a$ is a square in~$A$: that is, we have $a = x^2$, for some $x \in A$. Also, since $\Cay(D;S)$ is connected, we may write $x = s_1 s_2 \cdots s_n$ for some $s_1,\ldots,s_n \in S$. So $a = (s_1 s_2 \cdots s_n)^2$ can be written as a word in which every element of~$S$ occurs an even number of times. Since $\varphi$ is colour-permuting, this implies that $\varphi(a)$ can be written as a word in which, for each $s \in S$, the total number of occurrences of either $s$ or~$s^{-1}$ is even. Since $s$ and~$s^{-1}$ both either centralize~$A$ or invert it, this implies that $\varphi(a)$ centralizes~$A$. Since $A$ is self-centralizing in~$D$, we conclude that $\varphi(a) \in A$, as desired.
\end{proof}

\section{Groups of odd order} \label{OddOrderSect}

The following notation will be assumed throughout this section.

\begin{notation}
For a fixed Cayley graph $\Cay(G;S)$:
	\begin{itemize}
	\item $\A^0$ is the group of all colour-preserving automorphisms of $\Cay(G; S)$.
	\item $\widehat G$ is the subgroup of~$\A^0$ consisting of all left translations by elements of~$G$. (Although we do not need this terminology, it is often called the \emph{left regular representation} of~$G$.)
	\item $H_e$ is the stabilizer of the identity element~$e$ in~$\Cay(G;S)$, for any subgroup~$H$ of~$\A^0$.
	\end{itemize}
\end{notation}

\begin{rem} \label{Aff=N(G)}
It is well known (and very easy to prove) that a permutation of~$G$ is affine if and only if it normalizes~$\widehat G$ (see, for example \cite[Lem.~2]{Sehgal-NormCayRep}).
\end{rem}

\begin{lem} \label{StabIs2Grp}
$\A^0_e$ is a $2$-group.
\end{lem}

\begin{proof}
Let $\varphi \in \A^0_e$, so $\varphi$ is a colour-preserving automorphism of $\Cay(G;S)$ that fixes~$e$. If $C$ is any monochromatic cycle through~$e$, then either $\varphi$ is the identity on~$C$ or $\varphi$ reverses the orientation of~$C$. Therefore, $\varphi^2$ acts trivially on the union of all monochromatic cycles that contain~$e$. This implies that $\varphi^2$ acts trivially on all vertices at distance $\le 1$ from~$e$.

Repeating the argument shows that $\varphi^{2^k}$ acts trivially on all vertices at distance $\le k-1$ from~$e$. For $k$ larger than the diameter of $\Cay(G;S)$, this implies that $\varphi^{2^k}$ is trivial. So the order of~$\varphi$ is a power of~$2$.
\end{proof}

\begin{prop} \label{OddCCA->strong}
Let\/ $\Cay(G;S)$ be a connected Cayley graph on a group~$G$ of odd order. If\/ $\Cay(G;S)$ is CCA, then\/ $\Cay(G;S)$ is strongly CCA.
\end{prop}

\begin{proof} 
Let $\A^\bullet$ be the group of all colour-permuting automorphisms of $\Cay(G; S)$.
Since $\A^\bullet$ acts on the set of colours, and $\A^0$ is the kernel of this action (and the kernel of a homomorphism is always normal), it is obvious that $\A^0 \normal \A^\bullet$.
Also, since $\Cay(G;S)$ is CCA, we have $\widehat G \normal \A^0$ \ccf{Aff=N(G)}. Furthermore, $|G|$ is odd, $|\A^0_e|$ is a power of~$2$, and $\A^0 = \widehat G \cdot \A^0_e$. Therefore, $\widehat G$ is the (unique) largest normal subgroup of odd order in~$\A^0$. The uniqueness implies that $\widehat G$ is \emph{characteristic} in~$\A^0$. (That is, it is fixed by all automorphisms of~$\A^0$.) So $\widehat G$ is a characteristic subgroup of the normal subgroup~$\A^0$ of~$\A^\bullet$. Since every characteristic subgroup of a normal subgroup is normal \cite[Thm.~2.1.2(ii), p.~16]{Gorenstein-FiniteGrps}, this implies $\widehat G \normal \A^\bullet$. Therefore $G$ is strongly CCA \csee{Aff=N(G)}.
\end{proof}

Wreath products $\ZZ_m \wr \ZZ_n$ provide examples of non-CCA groups of odd order \csee{WreathProdNotCCA}. We will see in \cref{OddNotCCAHas} that the following slightly more general construction is essential for understanding many of the other non-CCA groups of odd order.

\begin{eg} \label{SemiWreath}
Let $\alpha$ be an automorphism of a group~$A$, and let $n \in \ZZ^+$. Then we can define an automorphism~$\widetilde \alpha$ of $A^n$ by
	$$ \widetilde \alpha(w_1,\ldots,w_n) = \bigl( \alpha(w_n),  w_1, w_2,\ldots, w_{n-1} \bigr) .$$
It is easy to see that the order of $\widetilde \alpha$ is $n$ times the order of~$\alpha$, so we may form the corresponding semidirect product $A^n \rtimes \ZZ_{n|\alpha|}$. Let us call this the \emph{semi-wreathed product} of $A$ by~$\ZZ_n$, with respect to the automorphism~$\alpha$, and denote it $A \wr_\alpha \ZZ_n$. 
(If $\alpha$ is the trivial automorphism, then this is the usual wreath product $A \wr \ZZ_n$.)

Negating the first coordinate, as in \cref{WreathProdNotCCA}, shows that if $n > 1$ and $A$~is abelian, but not an elementary abelian $2$-group, then $A \wr_\alpha \ZZ_n$ is not CCA.
\end{eg}

\begin{rem}
Because it may be of interest to find minimal examples, we point out that every semi-wreathed product of odd order satisfying the conditions in the 
final paragraph 
of \cref{SemiWreath} must contain a subgroup that is isomorphic to a semi-wreathed product $A \wr_\alpha \ZZ_q$, where $A$~is an elementary abelian $p$-group, $p$~and~$q$ are primes (not necessarily distinct), $\alpha$~is an automorphism of $q$-power order, and no nontrivial, proper subgroup of~$A$ is invariant under~$\alpha$.
\end{rem}

\begin{defn}[{}{\cite[p.~5]{Gorenstein-FiniteGrps}}]
Let $G$ be a group. For any subgroups $H$ and~$K$ of~$G$, such that $K \normal H$, the quotient $H/K$ is said to be a \emph{section} of~$G$.
\end{defn}

\begin{thm} \label{OddNotCCAHas}
Any non-CCA group of odd order has a section that is isomorphic to either:
	\begin{enumerate}
	\item \label{OddNotCCAHas-semiwreath}
	a semi-wreathed product $A \wr_\alpha \ZZ_n$ \csee{SemiWreath}, where $A$~is a nontrivial, elementary abelian group\/ \textup(of odd order\textup) and $n > 1$,
	or
	\item \label{OddNotCCAHas-21}
	the\/ \textup(unique\/\textup) nonabelian group of order\/~$21$.
	\end{enumerate}
\end{thm}

\begin{proof}
Assume $\Cay(G;S)$ is a connected Cayley graph on a group~$G$ of odd order that does not have a section as described in either \pref{OddNotCCAHas-semiwreath} or~\pref{OddNotCCAHas-21}.
We will show, by induction on the order, that if $\A$ is any subgroup of~$\A^0$ that contains~$\widehat G$, then $\widehat G$ is a normal subgroup of~$\A$.
(Then taking $\A = \A^0$ implies that $\Cay(G;S)$ is CCA \csee{Aff=N(G)}.)

It is important to note that this conclusion implies $\widehat G$ is a \emph{characteristic} subgroup of~$\A$ (because \cref{StabIs2Grp} implies that $\widehat G$ is the unique largest normal subgroup of odd order).
For convenience, we write $\widehat G \CHAR \A$ when $\widehat G$ is characteristic.

Let $\N$ be a minimal normal subgroup of~$\A$. Then $\N$ is either elementary abelian or the direct product of (isomorphic) nonabelian simple groups \cite[Thm.~2.1.5, p.~17]{Gorenstein-FiniteGrps}, and we consider the two possibilities as separate cases.

\setcounter{case}{0}

\begin{case} \label{Odd-NAbel}
Assume $\N$ is elementary abelian. 
\end{case}
Since the Sylow $2$-subgroup $\A_e$, being the stabilizer of a vertex, does not contain any normal subgroups of~$\A$, we know that $\N$ is not contained in a Sylow $2$-subgroup. Hence, $\N$ is not a $2$-group, so it must be a $p$-group for some odd prime~$p$.  Therefore, since $\widehat G$ is the largest normal subgroup of odd order, we have $\N \subseteq \widehat G$, so 
	$$ \text{$\N = \widehat N$, for some (elementary abelian) normal subgroup~$N$ of~$G$} .$$

Let $\N^+$ be the kernel of the action of~$\A$ on $G/N$, so $\A/\N^+$ is a group of colour-preserving automorphisms of $\Cay(G/N; \overline S)$, where $\overline S$ is the image of~$S$ in $G/N$.
Therefore, by induction on $|\A|$, we know that $\widehat G\N^+/\N^+$ is normal in $\A/\N^+$, so $\widehat G \N^+$ is normal in~$\A$. Then we may assume $\widehat G \N^+ = \A$, for otherwise, by induction on $|\A|$, we would know $\widehat G \CHAR \widehat G \N^+$, so $\widehat G \normal \A$, as desired. Since $|G|$ is odd, this implies that $\N^+$ contains a Sylow $2$-subgroup of~$\A$. In fact, since $\N^+$ is normal and all Sylow $2$-subgroups are conjugate, this implies that $\N^+$ contains every Sylow $2$-subgroup. In particular, it contains~$\A_e$. Therefore $\N^+ = \N \A_e$, so
	$$ \N \A_e \normal \A .$$
This means that $\A_e$ acts trivially on $G/N$, so, for every $s \in S \setminus N$, $\A_e$ preserves the orientation of every $s$-edge. (This uses the fact that, since $|s|$ is odd, $s \not\equiv s^{-1} \pmod N$ if $s \notin N$.) This implies:
	 \begin{align} \label{NotInNCCA}
	  \text{for $\varphi \in \A_e$, $g \in G$, and $x \in \langle S \setminus N \rangle$, we have $\varphi(gx) = \varphi(g) \, x$} 
	  . \end{align}

Let $(S \cap N)^{\langle S \setminus N\rangle} = \{\, g s g^{-1} \mid s \in S \cap N, \ g \in \langle S \setminus N\rangle\,\}$.
Now, suppose $t \in (S \cap N)^{\langle S \setminus N\rangle}$ and $h \in N$. There exists $s \in S \cap N$ and $x \in \langle S \setminus N \rangle$, such that $x s x^{-1} = t$. From \pref{NotInNCCA} and the fact that $\varphi$ is colour-preserving, we see that
	$$ \varphi(h \, t) = \varphi( h \, x s x^{-1}) = \varphi(h) \, x \, s^{\pm1} x^{-1} = \varphi(h) \, t^{\pm1} .$$
Hence, $\varphi|_N$ is a colour-preserving automorphism of 
	$$\Cay \Bigl( N; (S \cap N)^{\langle S \setminus N\rangle} \cup \bigl( \langle S \setminus N \rangle \cap N \bigr) \Bigr) .$$
Since $S$ generates~$G$, it is easy to see that this Cayley graph is connected. 

Note that $C_{\A_e}(\N)$ is normalized by both $\N$ and~$\A_e$, so it is a normal subgroup of $\N \A_e$. Therefore, it must be trivial (since the largest normal $2$-subgroup of $\N \A_e$ is characteristic, and is therefore normal in~$\A$, but the stabilizer $\A_e$ does not contain any nontrivial normal subgroups of~$\A$). So 
	\begin{align} \label{AeFaithful}
	\text{$\A_e$ acts faithfully by conjugation on~$\N$.}
	\end{align}

Also, we know that $\varphi|_N$ is an automorphism of~$N$ (by \cref{Aff=N(G)}, since $\varphi$ normalizes~$\N = \widehat N$). Since, being a colour-preserving automorphism, $\varphi$ either centralizes or inverts every element of the generating set of~$N$, this implies that $\varphi^2|_N$ is trivial. Since this is true for every $\varphi \in \A_e$, we conclude that $\A_e$ acts on~$N$ via an elementary abelian $2$-group. From \pref{AeFaithful}, we conclude that $\A_e$ is elementary abelian. 	

We can think of $\N$ as a vector space over~$\ZZ_p$, and, for each homomorphism $\gamma \colon \A \to \{\pm1\}$, let 
	$$\N_\gamma = \{\, n \in \N \mid \text{$a n a^{-1} = \gamma(a) \, n$ for all $a \in \A_e$} \,\} .$$
(This is called the ``weight space'' associated to~$\gamma$.)
Since every linear transformation satisfying $T^2 = I$ is diagonalizable, and $\A_e$ is commutative, the elements of~$\A_e$ can be simultaneously diagonalized.  This means that if we let $\Gamma = \bigl\{\,\gamma \mid \N_\gamma \neq \{0\} \,\bigr\}$, then, since eigenspaces for different eigenvalues are always linearly independent, we have $\N = \bigoplus_{\gamma \in \Gamma} \N_\gamma$. This direct-sum decomposition is canonically defined from the action of~$\A_e$ on~$\N$. Since $\widehat G$ acts on $\N\A_e$ (by conjugation), we conclude that the action of~$\widehat G$ on~$\N$ by conjugation must permute the weight spaces. 
More precisely, there is an action of $G$ on~$\Gamma$, such that $\widehat g \N_\gamma \widehat g^{-1} = \N_{g\gamma}$ for all $g \in G$. Since $N$ is abelian, this factors through to a well-defined action of $G/N$ on~$\Gamma$.

If the $G$-action on~$\Gamma$ is trivial, then every weight space is $G$-invariant, which implies that the action of~$\widehat G$ on~$\N$ commutes with the action of~$\A_e$. Since $\A_e$ acts faithfully, we conclude that $\widehat G$ centralizes $\A_e \N/\N$; that is, $[\widehat G,\A_e] \subseteq \N \subseteq \widehat G$. So $\A_e$ normalizes~$\widehat G$, as desired.

We may now assume that the $G$-action is nontrivial, so there is some $g \in G$ with an orbit of some length $n > 1$ on~$\Gamma$. Let $\gamma_0$ be an element of this orbit, so $\widehat g^n$ normalizes~$\N_{\gamma_0}$.  
Since $S \setminus N$ generates $G/N$, we may assume $g \in S \setminus N$, so \pref{NotInNCCA} tells us that $\langle \widehat g \rangle \cap \N$ is centralized by~$\A_e$. However, the minimality of~$\N$ implies that $C_{\N}(\A_e) = \N \cap Z(\N \A_e)$ is trivial. Therefore, $\langle \N, \widehat g \rangle = \N \rtimes \langle \widehat g \rangle$ is a semidirect product. So 
	$$\langle \N_{\gamma_0} , \widehat g \rangle = \left( \bigoplus\nolimits_{\gamma \in \langle g \rangle \gamma_0} \N_\gamma \right) \rtimes \langle \widehat g \rangle .$$
Then modding out $C_{\langle \widehat g \rangle}(\N_{\gamma_0})$ yields a section of~$\widehat G$ that is isomorphic to $\N_{\gamma_0} \wr_\alpha \ZZ_n$, where $\alpha$ is the automorphism of~$\N_{\gamma_0}$ induced by the conjugation action of~$\widehat g^n$. So $G$ has a semi-wreathed section, as described in \pref{OddNotCCAHas-semiwreath}.
This completes the proof of this case.

\begin{case} \label{Odd-NSimple}
Assume $\N = \L_1 \times \cdots \times \L_r$, where each $\L_i$~is a nonabelian simple group, and $\L_i \iso \L_1$ for all~$i$. 
\end{case}
We know that $\A = \widetilde G \A_e$, $\A_e$~is a $2$-group, and $|G|$ is odd, so $\widehat G$ is a $2$-complement in~$\A$. (By definition, this means that $|\widehat G|$ is odd and $|\A: \widehat G|$ is a power of~$2$ \cite[p.~88]{Isaacs-FiniteGrps}.) So $\L_1$~is a nonabelian simple group that has a $2$-complement (namely, $\widehat G \cap \L_1$). By using the Classification of Finite Simple Groups, it can be shown that this implies $\L_1 \iso \PSL(2,p)$, for some Mersenne prime $p \ge 7$ (see \cite[Thm.~1.3]{MartinezPerezWillems-TIProblem}). 

Note that $\A_e \cap \L_i$ is a Sylow $2$-subgroup of $\L_i \iso \PSL(2,p)$. Therefore, it is dihedral \cite[Lem.~15.1.1(iii)]{Gorenstein-FiniteGrps} and has order $p+1$ (because $p$~is a Mersenne prime). 
Let 
	\begin{itemize}
	\item $\C_i$ be the unique cyclic subgroup of order $(p+1)/2$ in $\A_e \cap \L_i$, 
	\item $\C_i^2$ be the unique subgroup of index~$2$ in~$\C_i$,
	and
	\item $\C^2 = \C_1^2 \times \cdots \times \C_r^2 \subset \A_e \cap (\L_1 \times \cdots \times \L_r)$.
	\end{itemize}
Since every element of~$\C_i$ is a colour-preserving automorphism, it either fixes or inverts each element of~$S$, so we know that $\C_i^2$ fixes every element of~$S$. Since stabilizers are conjugate, this implies $\widehat s^{-1} \C^2 \widehat s \subseteq \A_e$, for every $s \in S$. 
We must have $p > 7$, for otherwise $\widehat G \cap \L_i$, being the $2$-complement of $\PSL(2,7)$, would be the nonabelian group of order~$21$, as in \pref{OddNotCCAHas-21}. 
This implies that $\C_i^2$ is the unique cyclic subgroup of order $(p+1)/4$ in the dihedral group $\A_e \cap \L_i$, so we must have $\widehat s^{-1} \C^2 \widehat s = \C^2$, which means that $\widehat s$ normalizes~$\C^2$. 
Since this holds for every $s$ in the generating set~$S$, we conclude that $\widehat G$ normalizes~$\C^2$. 


Note that $\A_e$ normalizes $\A_e \cap \N$, and that $\C^2 \CHAR \A_e \cap \N$ (since, as was mentioned above, $\C_i$ is the unique cyclic subgroup of its order in $A_e \cap \L_i$). Therefore, $\C^2 \normal \A_e$. We conclude that $\C^2$ is normal in $\widehat G \A_e = \A$. 
So $\C_1^2 = \C^2 \cap \L_1$ is normal in~$\L_1$, contradicting the fact that $\L_1$ is simple.
\end{proof}

\begin{lem} \label{AssumeSPrimePower}
A group $G$ is strongly CCA\/ \textup(or CCA\/\textup) if and only if, for every generating set~$S$ of~$G$ such that every element of~$S$ has prime-power order, the Cayley graph\/ $\Cay(G;S)$ is strongly CCA\/ \textup(or CCA\/\textup).
\end{lem}

\begin{proof}
Suppose $\varphi$ is a colour-permuting automorphism of some connected Cayley graph $\Cay(G;S)$. There is a permutation $\pi$ of~$S$, such that $\varphi( gs ) = \varphi(g) \, \pi(s)^{\pm1}$, for all $g \in G$ and $s \in S$. (Furthermore, if $\varphi$ is colour-preserving, then $\pi$ can be taken to be the identity permutation.)
By induction on~$k$, this implies $\varphi(gs^k) = \varphi(g) \, \pi(s)^{\pm k}$, for all $k \in \ZZ$. Hence, if we let $S^* = \{\, s^k \mid s \in S,\ k \in \ZZ \,\}$, then $\varphi$ is a colour-permuting automorphism of $\Cay(G;S^*)$. Now, let
	$$ S_0 = \{\, t \in S^* \mid \text{$|t|$ is a prime-power} \,\} .$$
Then $\varphi$ is a colour-permuting automorphism of $\Cay(G;S_0)$, and $S_0$ generates~$G$, since every element~$s$ of the generating set~$S$ can be written as a product of elements of $\langle s \rangle$ that have prime-power order \cite[Thm.~1.3.1(iii), p.~9]{Gorenstein-FiniteGrps}, and therefore belong to~$S_0$. (Furthermore, $\varphi$ is colour-preserving if the permutation $\pi$ is the identity permutation.)
\end{proof}

\begin{lem} \label{ModCyclic}
Suppose 
	\begin{itemize}
	\item $C$ is a cyclic, normal subgroup of a group~$H$,
	\item $|C|$ is relatively prime to $|H: C|$,
	\item no element of $H \setminus C$ centralizes~$C$,
	and
	\item $\alpha$ is an automorphism of~$H$.
	\end{itemize}
Then $\alpha(h) \in h C$, for every $h \in H$.
\end{lem}

\begin{proof}
Since no other subgroup of~$H$ has the same order as~$C$, we know that $\alpha|_C$ is an automorphism of~$C$, so there exists $r \in \ZZ$, such that $\alpha(c) = c^r$, for every $c \in C$. Then, for every $h \in H$ and $c \in C$, we have
	$$ \alpha(h) \, c^r \, \alpha(h)^{-1} 
	= \alpha(h) \, \alpha(c) \, \alpha(h)^{-1} 
	= \alpha(hch^{-1})
	= (hch^{-1})^r
	= hc^rh^{-1} ,$$
so $h^{-1} \, \alpha(h)$ centralizes~$C$. By assumption, this implies $h^{-1} \, \alpha(h)  \in C$, as desired.
\end{proof}

\begin{cor} \label{OrderOfNonCCA}
The following are equivalent:
\noprelistbreak
	\begin{enumerate}
	\item \label{OrderOfNonCCA-NotCCA}
	There is a group of order~$n$ that is not CCA.
	\item \label{OrderOfNonCCA-NotStrong}
	There is a group of order~$n$ that is not strongly CCA.
	\item \label{OrderOfNonCCA-Divisible}
	$n \ge 8$, and $n$~is divisible by either $4$, $21$, or a number of the form $p^q \cdot q$, where $p$ and~$q$ are primes\/ \textup(not necessarily distinct\/\textup) and $p$~is odd.
	\end{enumerate}
\end{cor}

\begin{proof}
($\ref{OrderOfNonCCA-NotCCA} \Rightarrow \ref{OrderOfNonCCA-NotStrong}$) 
Obvious.

($\ref{OrderOfNonCCA-Divisible} \Rightarrow \ref{OrderOfNonCCA-NotCCA}$) 
If $n$ is divisible by~$4$, then there is a generalized dicyclic group of order~$n$, which is not CCA \fullcsee{4NotCCA}{dicyclic}. The nonabelian group of order 21 and the wreath product $\ZZ_p \wr \ZZ_q$ (which is of order $p^q \cdot q$) are not CCA \csee{21NotCCA,WreathProdNotCCA}. Taking an appropriate direct product yields a non-CCA group whose order is any multiple of these \csee{NotCCAxAny}.

($\ref{OrderOfNonCCA-NotStrong} \Rightarrow \ref{OrderOfNonCCA-Divisible}$)
Assume there is a group~$G$ of order~$n$ that is not strongly CCA, but $n$ is not divisible by $4$, $21$, or a number of the form $p^q \cdot q$.
From \cref{OddNotCCAHas}, we see that $n$ is even. 
(Otherwise, $n = |G|$ is divisible by the order of a semi-wreathed product $|A \wr_\alpha \ZZ_k|$. If we let $p$ and~$q$ be prime divisors of $|A|$ and~$k$, respectively, then $|A \wr_\alpha \ZZ_k| = |A|^k \cdot k$ is a multiple of $p^q \cdot q$.)
Furthermore, $n$ must be square-free, for otherwise it is a multiple of either~$4$ or $p^2 \cdot 2$, for some prime~$p$. Therefore, $G$ is a semidirect product $\ZZ_k \rtimes \ZZ_\ell$ \cite[Cor.~9.4.1]{Hall-ThyGrps}.

We may assume the centre of~$G$ is trivial, for otherwise we can write $G$ as a nontrivial direct product, so \cref{GxHPrime} (and induction on~$n$) implies that $G$ is CCA. Therefore, $k$ is odd (so $\ell$ is even), so we may write $G = \ZZ_k \rtimes (\ZZ_m \times \ZZ_2)$, and $\ZZ_m \times \ZZ_2$ acts faithfully on $\ZZ_k$. 
Let $H = \ZZ_k \rtimes \ZZ_m$, so $|H| = km$ is odd, and $H$ is the (unique) subgroup of index~$2$ in~$G$.

Let $\varphi$ be a colour-permuting automorphism of a connected Cayley graph $\Cay(G;S)$. (We wish to show that $\varphi$ is affine.)
There is no harm in assuming that every element of $S$ has prime order \csee{AssumeSPrimePower}. 

\setcounter{case}{0}

\begin{case} \label{OrderOfNonCCAPf-PresCase}
Assume $\varphi$ is colour-preserving.
\end{case}
Fix some $t \in S$ with $|t| = 2$. 
We claim we may assume that $t$~is the only element of order~$2$ in~$S$, and that $H = \langle S \setminus \{t\} \rangle$. 
To see this, let 
	\begin{itemize}
	\item $T$ be the set of all elements of order~$2$ in~$S$, 
	and
	\item $S' = \{t\} \cup \{\, uv \mid u,v \in T, u \neq v \,\} \cup (S \setminus T)$.
	\end{itemize}
It is easy to see that $\varphi$ is a colour-preserving automorphism of the connected Cayley graph $\Cay(G; S')$, and that $G = \langle S \setminus \{t\} \rangle \langle t \rangle$. This establishes the claims.

From \cref{OddNotCCAHas} (and the fact that $|H|$ is odd), we know that $\varphi|_H$ is affine. By composing with a left translation, we may assume that $\varphi$ fixes~$e$. Then $\varphi|_H$ is a group automorphism. 
By composing with an automorphism of $\ZZ_k \rtimes (\ZZ_m \times \ZZ_2)$ of the form $(x,y,z) \mapsto (x^r, y, z)$, we may assume $\varphi|_{\ZZ_k}$ is the identity map. Also, since $\varphi(s) \in \{s^{\pm1}\}$ for every $s \in S$, and $|H/\ZZ_k| = m$ is odd, \cref{ModCyclic} implies that $\varphi$ also fixes every element of $(S \cap H) \setminus \ZZ_k$. Hence, $\varphi|_H$ is an automorphism that fixes every element of a generating set, so $\varphi(h) = h$ for every $h \in H$. Since $\varphi(ht) = \varphi(h) \, t = ht$, for all $h \in H$ (because $\varphi$ is colour-preserving and $t = t^{-1}$), we conclude that $\varphi$ fixes every element of~$G$, and is therefore affine, as desired.

\begin{case}
The general case.
\end{case}
From \cref{OrderOfNonCCAPf-PresCase}, we see that $G$ is CCA, so $\widehat G \normal \A^0$. Hence, $\widehat H \CHAR \A^0$ (since it is the unique largest normal subgroup of odd order), so $\varphi$ normalizes~$\widehat H$. This implies that the restriction of~$\varphi$ to~$H$ is an automorphism of~$H$.

For each $s \in S$, let $\widetilde s = \varphi(s) \in S$. 
To prove that $\varphi$ is affine, it suffices to show $\varphi(xs) = \varphi(x) \, \widetilde s$ for all $x \in G$ and $s \in S$ \fullcsee{CCARems}{Directed}. If this is not the case, then, since $\varphi$ is colour-permuting, there must be some~$x$, such that $\varphi(xs) = \varphi(x) \, \widetilde s^{-1}$ (and $\widetilde s^{-1} \neq \widetilde s$, which means $|s| \neq 2$). This will lead to a contradiction.

Since $|s| \neq 2$, and we have assumed that every element of~$S$ has prime order (by \cref{AssumeSPrimePower}), we see that $s \in H$. Then, since $\varphi|_H$ is an automorphism, but
	$$ \varphi(xs) = \varphi(x) \, \widetilde s^{-1} 
	= \varphi(x) \, \varphi(s)^{-1} \neq \varphi(x) \, \varphi(s) ,$$
we must have $x \notin H$. Since $H$ has only two cosets, and there must be some element of~$S$ that is not in~$H$, this implies that we may assume $x \in S$, after multiplying on the left by an appropriate element of~$H$ (and using the fact that $\varphi$ normalizes~$\widehat H$). Note that, since $x \notin H$, and every element of~$S$ has prime order, this implies $|x|  = 2$. So the order of $\widetilde x$ is also~$2$, which implies $\widetilde x \notin H$ (since $|H|$ is odd).

Since $\varphi$ is colour-permuting, we have
	$$ \varphi(xs) = \varphi( {}^x \! s \, x) = \varphi( {}^x \! s) \widetilde x .$$
Also, by the choice of $x$ and~$s$, we have
	$$ \varphi(xs) = \varphi(x) \, \widetilde s^{-1} = \widetilde x \, \widetilde s^{-1} . $$
Therefore 
	$$ \varphi( {}^x \! s) = {}^{\widetilde x} \widetilde s^{-1} .$$
Since $\ZZ_m$ acts faithfully on~$\ZZ_k$, we have $\alpha(h) \equiv h \pmod{\ZZ_k}$, for every automorphism~$\alpha$ of~$H$ \csee{ModCyclic}. Since $\varphi$ and conjugation by~$x$ are automorphisms of~$H$, this implies $s \equiv s^{-1} \pmod{\ZZ_k}$. Since $|s|$ is odd, we conclude that $s \in \ZZ_k$.

Then, since the automorphism group of a cyclic group is abelian, we have
	$$ \varphi( {}^x \! s) =  {}^x \! \varphi( s) =  {}^x \widetilde s , $$
so $x^{-1} \widetilde x$ must invert $\widetilde s$. But this is impossible, because, as was mentioned above, $x$ and~$\widetilde x$, being of order~$2$, cannot be in~$H$, so they are both in the other coset of~$H$, so $x^{-1} \widetilde x \in H$ has odd order. This contradiction completes the proof that $\varphi$ is affine.
\end{proof}

\begin{rem} \label{OrderOfNonCCADivisiblePNotOdd}
It is not necessary to assume $p$ is odd in the statement of \fullcref{OrderOfNonCCA}{Divisible}, because $2^q \cdot q$ is divisible by~$4$, which is already in the list of divisors.
\end{rem}

\section{Groups of small order} \label{SmallGrpSect}

In this section, we briefly explain which groups of order less than~$32$ are CCA (or strongly CCA). First, note that almost all of the abelian ones are strongly CCA:

\begin{prop}[cf.\ \cref{abelian}]
An abelian group of order less than\/ $32$ is not strongly CCA if and only if it is either
	\begin{itemize}
	\item $\ZZ_2 \times \ZZ_4$ \textup(of order~$8$\textup),
	\item $\ZZ_2 \times \ZZ_2 \times \ZZ_4$ \textup(of order~$16$\textup),
	or
	\item $\ZZ_2 \times \ZZ_3 \times \ZZ_4$ \textup(of order~$24$\textup).
	\end{itemize}
None of these are CCA.
\end{prop}

Also note that almost all of the groups whose order is not divisible by~$4$ are CCA:

\begin{prop}
The only groups that are not strongly CCA, and whose order is less than\/ $32$ and not divisible by\/~$4$ are:
	\begin{itemize}
	\item the wreath product\/ $\ZZ_3 \wr \ZZ_2$, which is isomorphic to $D_6 \times \ZZ_3$ and has order\/~$18$,
	and
	\item the nonabelian group of order\/~$21$.
	\end{itemize}
Neither of these is CCA.
\end{prop}

\begin{proof}
For the groups of odd order, the conclusion is immediate from \cref{OddNotCCAHas} and \cref{21NotCCA} (see \cref{OddLess100} for a stronger result). \Cref{GxHPrime} deals with the groups $D_6 \times \ZZ_5$ and $D_{10} \times \ZZ_3$ of order~$30$. For all of the other groups of even order, it suffices to note that if $m$ is odd, then every generalized dihedral group of order~$2m$ is strongly CCA \csee{DihIff}.
\end{proof}

So it is surprising that very few of the remaining groups are strongly CCA:

\begin{prop} \label{StrongBy4}
The only nonabelian groups that are strongly CCA and whose order is less than\/ $32$ and  divisible by\/~$4$ are:
	\begin{itemize}
	\item the dihedral groups of order\/~$8$, $16$, and\/~$24$,
	\item the alternating group $A_4$, which is of order~$12$,
	\item another group of order 16, namely, the semidirect product 
		$$\ZZ_8 \rtimes \ZZ_2 = \langle\, x, a \mid x^8 = a^2 = e, \ a^{-1} x a = x^5\,\rangle ,$$
	and
	\item three additional groups groups of order 24, namely, $D_8 \times \ZZ_3$, $A_4 \times \ZZ_2$, and the semidirect product\/ $\ZZ_3 \rtimes \ZZ_8$ in which $\ZZ_8$ inverts~$\ZZ_3$.
	\end{itemize}
Furthermore, the only groups of order less than\/ $32$ that are CCA, but not strongly CCA, are:
	\begin{itemize}
	\item the dihedral groups $D_{12}$, $D_{20}$, and $D_{28}$,
	and
	\item the group $D_{12} \times \ZZ_2$, which is a generalized dihedral group of order 24.
	\end{itemize}
\end{prop}

\begin{proof}[Sketch of proof] 
The result can be verified by an exhaustive computer search, but we summarize a case-by-case analysis that can be carried out by hand, using the classification of groups of order less than 32.
Each group of such small order can be specified by its ``GAP Id\rlap,'' which is an ordered pair $[n,k]$, where $n$ is the order of the group, and $k$~is the id~number that has been assigned to that particular group (see \cite{Wiki-GrpsOfPartOrder}, for example).

Assume $G$ is nonabelian, $|G| < 32$, and $|G|$ is divisible by~$4$. 
We may assume that $G$ is neither generalized dicyclic, semidihedral, nor generalized dihedral, for otherwise \cref{4NotCCA}\textup(\ref{4NotCCA-dicyclic},\ref{4NotCCA-semidihedral}\textup) and \cref{GenDihedralCCA,DihIff} determine whether $G$ is CCA or strongly CCA.
\SeeAppendix{EasyOmit} 
By inspection of the list of groups of each order, we see that this leaves only thirteen possibilities for~$G$, and we consider each of these GAP Ids separately.
In most cases, \cref{notCCA-structure} implies that $G$ is not CCA.

	\begin{itemize}
	
	\item[{$[12,3]$}] $= A_4$. This group is strongly CCA, but we omit the proof.
	\SeeAppendix{IsStrong} 
	
	\item[{$[16,3]$}] $=\langle\, a,b,c \mid a^4=b^2=c^2=e, ab=ba, bc=cb, cac=ab\,\rangle$. 
	\Cref{notCCA-structure} applies with $S = \{a^{\pm1}, c\}$, $T = \{a^{\pm1}\}$, and $\tau = a^2 \in Z(G)$.

	\item[{$[16,6]$}] $=\langle\, a,x \mid a^8=x^2=e, x a x = a^5 \,\rangle
	 = \langle a \rangle \rtimes \langle x \rangle = \ZZ_8 \rtimes \ZZ_2$. 
	This group is strongly CCA, but we omit the proof.
	\SeeAppendix{IsStrong} 

	\item[{$[16,13]$}] $= \langle\, a,x,y \mid a^4=x^2=e, a^2=y^2, xax=a^{-1}, ay=ya, xy=yx \,\rangle$. \Cref{notCCA-structure} applies with $S = \{a^{\pm1}, x, y^{\pm1}\}$, $T = \{a^{\pm1}, y^{\pm1}\}$, and $\tau = a^2 \in Z(G)$.
	
	\item[{$[20,3]$}] $= \langle\, a,b \mid a^5=b^4=e, \ ba b^{-1}=a^2 \,\rangle$. \Cref{notCCA-structure} applies with $S=\{a^{\pm1},b^{\pm1}\}$, $T = \{b^{\pm1}\}$, and $\tau = b^2$ (which inverts~$a$).
	
	\item[{$[24,1]$}] $= \ZZ_3 \rtimes \ZZ_8$, where $\ZZ_8$ inverts~$\ZZ_3$.
	This group is strongly CCA, but we omit the proof.%
	\SeeAppendix{IsStrong} 
	
	\item[{$[24,3]$}] $= \mathrm{SL}(2,3) \iso Q_8 \rtimes \ZZ_3  
	= \langle i,j \rangle \rtimes \langle a \rangle$, where $a i a^{-1} = j$ and $a^{-1} i a = ij $. 
	\Cref{notCCA-structure} applies with $S=\{ i^{\pm1}, a^{\pm1}\}$, $T = \{i^{\pm1}\}$, and $\tau = i^2 \in Z(G)$.
	
	\item[{$[24,5]$}] $= S_3 \times \ZZ_4$. \Cref{notCCA-structure} applies with 
	$T =  \{(1,2)\} \times \{\pm1\}$, $S = \bigl\{ \bigl( (2,3), 0 \bigr) \bigr\} \cup T$, and $\tau = (e,2) \in Z(G)$.
	
	\item[{$[24,8]$}] $= \ZZ_3 \rtimes D_8 = \langle\, a,b,c \mid a^3 = b^4 = c^2 = e, \ bab^{-1} = a^{-1}, \ ac = ca, \, cbc^{-1} = b^{-1} \,\rangle$.  \Cref{notCCA-structure} applies with $S = \{(ab)^{\pm1}, b^{\pm1}, c\}$, $T = \{(ab)^{\pm1}, b^{\pm1}\}$, and $\tau = b^2 \in Z(G)$.

	\item[{$[24,10]$}] $= D_8 \times \ZZ_3$. Since $D_8$ is strongly CCA \csee{DihIff}, the same is true for this group \csee{GxHPrime}.

	\item[{$[24,11]$}] $= Q_8 \times \ZZ_3$. This is not CCA, since $Q_8$ is not CCA \csee{Z4xZ2andQ8,NotCCAxAny}.
	
	\item[{$[24,12]$}] $= S_4$.
	Let $a = (1,2,3,4)$ and $b = (1,2,4,3)$, so
	\cref{notCCA-structure} applies, with $S = \{a^{\pm1},b^{\pm1}\}$, $T = \{a^{\pm1}\}$, and $\tau = a^2 = (1,3)(2,4)$, which inverts~$b$.
	
	\item[{$[24,13]$}] $= A_4 \times \ZZ_2$. 
	\SeeAppendix{IsStrong} 
	This group is strongly CCA, but we omit the proof. 
	\qedhere

	\end{itemize}
\end{proof}

The above results assume $|G| < 32$, but it is not difficult to treat considerably larger groups if we assume the order is odd:

\begin{prop} \label{OddLess100}
Let $G_{21} = \ZZ_7 \rtimes \ZZ_3$ be the\/ \textup(unique\textup) nonabelian group of order\/~$21$. Then the only groups of odd order less than\/ $100$ that are not strongly CCA are 
	$G_{21}$, 
	$G_{21} \times \ZZ_3$,
	and\/ $\ZZ_3 \wr \ZZ_3$.
\end{prop}

\begin{proof}
Suppose $G$ is a group of odd order, such that $G$ is not strongly CCA and $|G| < 100$. From \cref{OrderOfNonCCA}, we see that $|G|$ is divisible by either $21$ or $3^3 \cdot 3 = 81$. Since $|G| < 100$, this implies that $|G|$ is either $21$, $21 \times 3 = 63$, or $3^3 \cdot 3 = 81$. Also, $G$ must be nonabelian \csee{OrderOfNotCCAAbel}.
	\begin{itemize}
	\item The nonabelian group $G_{21}$ of order~$21$ is not CCA \csee{21NotCCA}.
	\item There are two nonabelian groups  of order $63$. One of them, the direct product $G_{21} \times \ZZ_3$, is not CCA \csee{NotCCAxAny}. The other is
$$ \ZZ_7 \rtimes \ZZ_9 = \langle\, x, a \mid x^7 = a^9 = e, \ a^{-1} x a = x^2 \,\rangle .$$
	This group is strongly CCA, but we omit the proof.
	\SeeAppendix{IsStrong} 

	\item \Cref{OddNotCCAHas} implies that $\ZZ_3 \wr \ZZ_3$ is the only non-CCA group of order~$81$ (see also \cref{WreathProdNotCCA}). \qedhere
	\end{itemize}
\end{proof}


\newpage 
\thispagestyle{empty}
\cleardoublepage
\thispagestyle{empty}

\makeatletter
\def\@oddrunninghead{Appendix:  Notes to aid the referee}
\makeatother

\begin{appendix}

\begin{center}
\large\bf
\hypertarget{AppendixForRef}{APPENDIX: Notes to aid the referee} 
\end{center}

\section{Dicyclic (and other) groups omitted from the proof of \cref{StrongBy4}}
\label{EasyOmit}
In order to keep the proof of \cref{StrongBy4} short, it has no discussion of groups that are generalized dicyclic, semidihedral, or generalized dihedral.
Here are additional details to show that no cases were missed.

\begin{itemize} \itemsep=-1pt

\item[\bf order 4:] \url{http://groupprops.subwiki.org/wiki/Groups_of_order_4}
	\item[] There are no nonabelian groups of order~$4$.
	
	\medbreak

\item[\bf order 8:] \url{http://groupprops.subwiki.org/wiki/Groups_of_order_8}
	\item[{$[8,1]$}] $= \ZZ_8$ is abelian
	\item[{$[8,2]$}] $= \ZZ_4 \times \ZZ_2$ is abelian
	\item[{$[8,3]$}] $= D_8$ is dihedral (and is \textbf{strongly CCA}, since $\ZZ_4$ does not have $\ZZ_2$ as a direct factor)
	\item[{$[8,4]$}] $= Q_8$ is dicyclic (so it is not CCA)
	\item[{$[8,5]$}] $= \ZZ_2 \times \ZZ_2 \times \ZZ_2$ is abelian
	
	\medbreak
	
\item[\bf order 12:] \url{http://groupprops.subwiki.org/wiki/Groups_of_order_12}
	\item[{$[12,1]$}] $= Q_{12} = \mathrm{Dic}\bigl( 3, \ZZ_6 \bigr)$ is dicyclic (so it is not CCA)
	\item[{$[12,2]$}] $= \ZZ_{12}$ is abelian
	\item[{$[12,3]$}] $= A_4$ is discussed in the proof of \cref{StrongBy4}. It is \textbf{strongly CCA}.
	\item[{$[12,4]$}] $= D_{12}$ is dihedral (and is \textbf{CCA but not strongly CCA}, since $\ZZ_6$ has $\ZZ_2$ as a direct factor)
	\item[{$[12,5]$}] $= \ZZ_6 \times \ZZ_2$ is abelian
	
	\medbreak

\item[\bf order 16:] \url{http://groupprops.subwiki.org/wiki/Groups_of_order_16}
	\item[{$[16,1]$}] $= \ZZ_{16}$ is abelian
	\item[{$[16,2]$}] $= \ZZ_4 \times \ZZ_4$ is abelian
	\item[{$[16,3]$}]  is discussed in the proof of \cref{StrongBy4}. It is not CCA.
	\item[{$[16,4]$}]  $= \ZZ_4 \rtimes \ZZ_4 = \mathrm{Dic}\bigl( (0,1), \ZZ_4 \times \ZZ_2 \bigr)$ is generalized dicyclic (so it is not CCA).
	\item[{$[16,5]$}] $=\ZZ_8 \times \ZZ_2$ is abelian
	\item[{$[16,6]$}]  is discussed in the proof of \cref{StrongBy4}. It is \textbf{strongly CCA}.
	\item[{$[16,7]$}] $= D_{16}$ is dihedral (and is \textbf{strongly CCA}, since $\ZZ_8$ does not have $\ZZ_2$ as a direct factor)
	\item[{$[16,8]$}] is semidihedral (so it is not CCA).
	\item[{$[16,9]$}] $= Q_{16} = \mathrm{Dic}\bigl( 4, \ZZ_8 \bigr)$ is dicyclic (so it is not CCA)
	\item[{$[16,10]$}] $=\ZZ_4 \times \ZZ_2 \times \ZZ_2$ is abelian
	\item[{$[16,11]$}] $= D_8 \times \ZZ_2$ is generalized dihedral over $\ZZ_4 \times \ZZ_2$ (and it is not CCA, since $\ZZ_4 \times \ZZ_2$ is not CCA).
	\item[{$[16,12]$}] $= Q_8 \times \ZZ_2 = \mathrm{Dic}\bigl( (2,0), \ZZ_4 \times \ZZ_2 \bigr)$ is generalized dicyclic (so it is not CCA).
	\item[{$[16,13]$}] is discussed in the proof of \cref{StrongBy4}. It is not CCA.
	\item[{$[16,14]$}] $=\ZZ_2 \times \ZZ_2 \times \ZZ_2 \times \ZZ_2$ is abelian

	\medbreak

\item[\bf order 20:] \url{http://groupprops.subwiki.org/wiki/Groups_of_order_20}
	\item[{$[20,1]$}] $= Q_{20}  = \mathrm{Dic}\bigl( 5, \ZZ_{10} \bigr)$ is dicyclic (so it is not CCA)
	\item[{$[20,2]$}] $= \ZZ_{20}$ is abelian
	\item[{$[20,3]$}] is discussed in the proof of \cref{StrongBy4}. It is not CCA.
	\item[{$[20,4]$}] $= D_{20}$ is dihedral (and is \textbf{CCA but not strongly CCA}, since $\ZZ_{10}$ has $\ZZ_2$ as a direct factor)
	\item[{$[20,5]$}] $= \ZZ_{10} \times \ZZ_2$ is abelian
	
	\medbreak
	
\item[\bf order 24:] \url{http://groupprops.subwiki.org/wiki/Groups_of_order_24}
	\item[{$[24,1]$}]  is discussed in the proof of \cref{StrongBy4}. It is \textbf{strongly CCA}.
	\item[{$[24,2]$}] $= \ZZ_{24}$ is abelian
	\item[{$[24,3]$}] is discussed in the proof of \cref{StrongBy4}. It is not CCA.
	\item[{$[24,4]$}] $= Q_{24}  = \mathrm{Dic}\bigl( 6, \ZZ_{12} \bigr)$ is dicyclic (so it is not CCA).
	\item[{$[24,5]$}]  is discussed in the proof of \cref{StrongBy4}. It is not CCA.
	\item[{$[24,6]$}] $= D_{24}$ is dihedral (and is \textbf{strongly CCA}, since $\ZZ_{12}$ does not have $\ZZ_2$ as a direct factor)
	\item[{$[24,7]$}] $= Q_{12} \times \ZZ_2 = \mathrm{Dic}\bigl( (3,0), \ZZ_6 \times \ZZ_2 \bigr)$ is generalized dicyclic (so it is not CCA).
	\item[{$[24,8]$}] is discussed in the proof of \cref{StrongBy4}. It is not CCA.
	\item[{$[24,9]$}] $= \ZZ_6 \times \ZZ_4$ is abelian
	\item[{$[24,10]$}] is discussed in the proof of \cref{StrongBy4}. It is \textbf{strongly CCA}.
	\item[{$[24,11]$}]  is discussed in the proof of \cref{StrongBy4}. It is not CCA.
	\item[{$[24,12]$}]  is discussed in the proof of \cref{StrongBy4}. It is not CCA.
	\item[{$[24,13]$}] is discussed in the proof of \cref{StrongBy4}. It is \textbf{strongly CCA}.
	\item[{$[24,14]$}] $= D_{12} \times \ZZ_2$ is generalized dihedral over $\ZZ_6 \times \ZZ_2$ (and is \textbf{CCA but not strongly CCA}, since $\ZZ_6 \times \ZZ_2$ obviously has $\ZZ_2$ as a direct factor)
	\item[{$[24,15]$}] $=\ZZ_3 \times \ZZ_2 \times \ZZ_2 \times \ZZ_2$ is abelian
	
	\medbreak

\item[\bf order 28:] There does not seem to be a web page for the groups of this order, but it is easy to classify the nonabelian ones, by noting that Sylow's Theorem implies the Sylow $7$-subgroup is normal. This implies $G = \ZZ_7 \rtimes H$, where $|H| = 4$ (so $H$ is either $\ZZ_4$ or $\ZZ_2 \times \ZZ_2$).

If $H = \ZZ_4$, then the only nonabelian semidirect product is the dicyclic group $Q_{28} = \mathrm{Dic}\bigl( 7, \ZZ_{14} \bigr)$ (which is not CCA).

If $H = \ZZ_2 \times \ZZ_2$, then the only nonabelian semidirect product is the dihedral group $D_{28}$. (It is \textbf{CCA, but not strongly CCA}, since $\ZZ_{14}$ has $\ZZ_2$ as a direct factor.)

\end{itemize}

\newpage
\section{Strongly CCA groups omitted from the proof of \cref{StrongBy4}} \label{IsStrong}
This appendix proves that certain small groups (listed in \cref{SmallStrong,63isCCA}) are strongly CCA. These verifications were omitted from our proofs of \cref{StrongBy4,OddLess100}.
The following simple observation will play a key role.

\begin{lem} \label{WordMakesAut}
Let
	\begin{itemize}
	\item $\varphi$ be a colour-permuting automorphism of a Cayley graph $\Cay(G;S)$, such that $\varphi(e) = e$,
	\item $\widetilde a = \varphi(a)$ and $\widetilde b = \varphi(b)$, for some $a,b \in S$,
	\item $\tau(v) \in \{\pm1\}$, such that $\varphi(va) = \varphi(v) \, \widetilde a^{\tau(v)}$, for all $v \in G$,
	and
	\item $k_1,k_2,\ldots,k_{2r} \in \ZZ \smallsetminus \{0\}$, such that
		$a^{k_1} b^{k_2} a^{k_3} \cdots b^{k_{2r}} = e$ \textup(and $r \ge 2$\textup).
	\end{itemize}
If $\epsilon_1 = \epsilon_3$ and $\epsilon_2 = \epsilon_4$, for all $\epsilon_1,\ldots,\epsilon_{2r} \in \{\pm1\}$, such that
	$\widetilde a^{\epsilon_1 k_1} \widetilde b^{\epsilon_2 k_2} \cdots \widetilde b^{\epsilon_{2r}k_{2r}} = e$,
then $\varphi(v a) = \varphi(v)\, \widetilde a$ and $\varphi(v b) = \varphi(v)\, \widetilde b$, for all $v \in \langle a^{k_3} , b^{k_2} \rangle$.
\end{lem}

\begin{proof}
Since $\varphi$ is colour-permuting, there exist $\sigma,\tau \colon G \to \{\pm1\}$, such that 
	$$ \text{$\varphi(va) = \varphi(v) \, \widetilde a^{\sigma(v)}$
	and
	$\varphi(va) = \varphi(v) \, \widetilde b^{\tau(v)}$
	 for all $v \in G$} .$$
We wish to show $\sigma(v) = \tau(v) = 1$ for all $v \in \langle a^{k_3} , b^{k_2} \rangle$. Since $\sigma(e) = \tau(e) = 1$, it suffices to show that $\sigma(v b^{k_2}) = \tau(v a^{k_3}) = \tau(v)$ for all $v \in G$.

The two parts of the proof are very similar, so we show only that $\sigma(v b^{k_2}) = \sigma(v)$. The relation $a^{k_1} b^{k_2} a^{k_3} \cdots b^{k_{2r}} = e$ represents a closed walk starting at~$v$ (or at any other desired vertex). Applying $\varphi$ yields a closed walk starting at~$\varphi(v)$. Since $\varphi$ is colour-permuting, this closed walk corresponds to a relation of the form $\widetilde a^{\epsilon_1 k_1} \widetilde b^{\epsilon_2 k_2} \cdots \widetilde b^{\epsilon_{2r}k_{2r}} = e$, with $\epsilon_i \in \{\pm1\}$. By assumption, we must have $\epsilon_1 = \epsilon_3$. Therefore 
	$$\sigma(a^{k_1} b^{k_2}) = \epsilon_3 = \epsilon_1 = \sigma(v) .$$ 
This establishes the desired conclusion, since $\sigma(v) = \sigma(v a^{k_1})$, and $v a^{k_1}$ is an arbitrary element of~$G$.
\end{proof}

\begin{eg} \label{SmallStrong}
The groups $[12,3]$, $[16,6]$, $[24,1]$, and $[24,13]$ from the proof of \cref{StrongBy4} are strongly CCA.
\end{eg}

\begin{proof}
We consider each of the four groups individually; for convenience, let $G$ be the group under consideration. Suppose $\varphi$ is a colour-permuting automorphism of a connected Cayley graph $\Cay(G;S)$, such that $\varphi(e) = e$, and let $\widetilde s = \varphi(s)$, for each $s \in S$.
We wish to show $\varphi \in \Aut G$.

\begin{case*}
Assume $G = [12,3]$.
\end{case*}
 Let $a \in S$ with $|a| = 3$, and let $N$ be the (unique) subgroup of order~$4$ in~$G$.

Assume, for the moment, that there exists $b \in S \cap N$ (so $|b| = 2$). Then $(ab)^3 = e$.
Suppose $i,j,k \in \{\pm1\}$, with
	$$ e = \widetilde a^i \, \widetilde b \, \widetilde a^j \, \widetilde b \, \widetilde a^k \, \widetilde b
	\equiv \widetilde a^{i + j + k} \pmod{N} ,$$
so $i + j + k \equiv 0 \pmod{3}$. Since $i,j,k \in \{\pm1\}$, this implies $i = j = k$. 
We conclude from \cref{WordMakesAut} that $\varphi(v s) = \varphi(v) \, \widetilde s$, for all $v \in \langle a, b \rangle = G$ and $s \in \{a^{\pm1}, b \}$, so $\varphi \in \Aut G$.

We may now assume $|s| = 3$ for all $s \in S$. Let $b \in S \smallsetminus \langle a \rangle$. We may assume $a \equiv b \pmod{N}$, by replacing $b$ with its inverse if necessary. Write $\widetilde b = \widetilde a^r x$, with $r \in \{\pm1\}$ and $x \in N$. Note that $(a^{-1} b)^2 = e$.
Suppose $i,j,k,\ell \in \{\pm1\}$, with
	$$ e = \widetilde a^{-i} \, \widetilde b^j \, \widetilde a^{-k} \, \widetilde b^\ell
	= \widetilde a^{-i+rj-k+r\ell} \cdot 
	\begin{cases}
	(\widetilde a^{k - r\ell}x \widetilde a^{-k+ r\ell}) x
		& \text{if $j = \ell = 1$} , \\
	(\widetilde a^k x \widetilde a^{-k}) x
		& \text{if $j = -1$ and $\ell = 1$} , \\
	\widetilde a^{-r\ell} \, (\widetilde a^k x \widetilde a^{-k}) x \, \widetilde a^{r\ell}
		& \text{if $j = 1$ and $\ell = -1$} , \\
	\widetilde a^{-r\ell} \, (\widetilde a^{k - rj}x \widetilde a^{-k + rj}) x \, \widetilde a^{r\ell}
		& \text{if $j = \ell = -1$}
	. \end{cases} $$
Since the component in~$N$ must be trivial, and no nontrivial power of~$\widetilde a$ centralizes~$x$, we see that we must have $j = \ell$ and $k = rj = r\ell$. Then, since the exponent of~$\widetilde a$ must be~$0$, this implies $i = k$.
We conclude from \cref{WordMakesAut} that $\varphi(v s) = \varphi(v) \, \widetilde s$, for all $v \in \langle a, b \rangle = G$ and $s \in \{a^{\pm1}, b^{\pm 1} \}$, so $\varphi \in \Aut G$.

\begin{case*}
Assume $G = [16,6]$.
\end{case*}
 Let $a \in S$ with $|a| = 8$. Let $b \in S \smallsetminus \langle a \rangle$.

Assume, for the moment, that $|b| = 8$. Write $b^2 = a^{2r}$, for some odd~$r$. Then we must have $\widetilde b^2 = \widetilde a^{2r}$. This implies that if $i,j \in \{\pm1\}$, such that 
	$$ e = \widetilde b^{2i} \, \widetilde a^{-2rj}, $$
then $i = j$ (since $|\widetilde b^2| = |b^2| > 2$). We conclude (much as in \cref{WordMakesAut}) that $\varphi(v s) = \varphi(v) \, \widetilde s$, for all $v \in \langle a, b \rangle = G$ and $s \in \{a^{\pm1}, b^{\pm 1} \}$, so $\varphi \in \Aut G$.

We may now assume $|b| \in \{2,4\}$, so $b^2 \in \langle a^4 \rangle$. Note that, since $b \notin \langle a \rangle$, we have $b a b^{-1} a^3 = e$. Suppose $i,j,k,\ell \in \{\pm1\}$, with
	$$ e = \widetilde b^i \, \widetilde a^j \, \widetilde b^{-k} \, \widetilde a^{3\ell}
	= \widetilde b^{i-k} \, \widetilde a^{5j+3\ell}
	\equiv a^{j-\ell} \pmod{\langle \widetilde a^4 \rangle}
	,$$
so $j = \ell$. Then we must also have $i = k$.
We conclude from \cref{WordMakesAut} that $\varphi(v s) = \varphi(v) \, \widetilde s$, for all $v \in \langle a, b \rangle = G$ and $s \in \{a^{\pm1}, b^{\pm 1} \}$, so $\varphi \in \Aut G$.

\begin{case*}
Assume $G = [24,1]$.
\end{case*}
Let $a \in S$ with $|a| = 8$, and let $b \in S$, such that $b \notin \langle a \rangle$. Write $\widetilde b = \widetilde a^r x$, where $\langle x \rangle = \ZZ_3$. 
We may assume $\widetilde b$ has prime-power order \csee{AssumeSPrimePower}, and we know that $\widetilde a^2$ centralizes~$x$, so either $r$~is odd or $r = 0$.

Assume, for the moment, that $r = 0$, which means $\langle \widetilde b \rangle = \ZZ_3 = \langle b \rangle$. Then $a$ inverts~$b$, so $aba^{-1}b = e$. Suppose $i,j,k,\ell \in \{\pm1\}$, with
	$$ e = \widetilde a^{i} \, \widetilde b^j \, \widetilde a^{-k} \, \widetilde b^\ell
	 = \widetilde a^{i- k} \, \widetilde b^{-j + \ell} 
	 .$$
Since the exponents of~$\widetilde a$ and~$\widetilde b$ must be~$0$, we have $i = k$ and $j = \ell$.
We conclude from \cref{WordMakesAut} that $\varphi(v s) = \varphi(v) \, \widetilde s$, for all $v \in \langle a, b \rangle = G$ and $s \in \{a^{\pm1}, b^{\pm 1} \}$, so $\varphi \in \Aut G$.

We may now assume that $r$ is odd.
The proof of \cref{AssumeSPrimePower} shows there is no harm in replacing $b$ with a power that is relatively prime to~$8$, so we may assume $r  = 1$.
Since $a^2 \in Z(G)$, we have $a^2 b a^{-2} b^{-1} = e$. Suppose $i,j,k,\ell \in \{\pm1\}$, with
	$$ e = \widetilde a^{2i} \, \widetilde b^j \, \widetilde a^{-2k} \, \widetilde b^{-\ell}
	 = \widetilde a^{2i- 2k} \, \widetilde b^{j- \ell} 
	 \equiv \widetilde a^{j - \ell} \pmod{\langle \widetilde a^4, x \rangle}
	 .$$
Then $j = \ell$. Therefore $\widetilde a^{2i- 2k} = e$, so $i = k$. 
For $v \in G$ with $\varphi(va) = \varphi(v) \, \widetilde a$, we conclude from the proof of \cref{WordMakesAut} that $\varphi(vba) = \varphi(vb) \, \widetilde a$. In addition, interchanging the roles of $a$ and~$b$ tells us that if $\varphi(vb) = \varphi(v) \, \widetilde b$, then $\varphi(v a b) = \varphi(va) \, \widetilde b$. We conclude that $\varphi(v s) = \varphi(v) \, \widetilde s$, for all $v \in \langle a, b \rangle = G$ and $s \in \{a^{\pm1}, b^{\pm 1} \}$, so $\varphi \in \Aut G$.

\begin{case*}
Assume $G = [24,13]$.
\end{case*}
 We may assume $|s| \in \{2,3\}$, for all $s \in S$ \csee{AssumeSPrimePower}. Let $a \in S$ with $|a| = 3$. 
Choose $b \in S$, such that $b \notin A_4$. Since every element of order~$3$ is contained in~$A_4$, we must have $|b| = 2$. 

Assume, for the moment, that $\langle a, b \rangle = G$.  
Note that $(a b a^{-1} b)^2 = e$, and, for convenience, let $\widetilde b_m = \widetilde a^{-m} \, \widetilde b \, \widetilde a^m$ for $m \in \ZZ$.
Suppose $i,j,k,\ell \in \{\pm1\}$, with
	$$ e = \widetilde a^i \, \widetilde b \, \widetilde a^{-j} \, \widetilde b \, \widetilde a^k \, \widetilde b \, \widetilde a^{-\ell} \, \widetilde b 
	= \widetilde a^{i - j + k - \ell} \cdot  
	\widetilde b_{-j + k - \ell} \, \widetilde b_{k - \ell} \,\widetilde b_{-\ell} \,\widetilde b .$$
This implies $k = \ell$, for otherwise $0$, $-\ell$, and $k - \ell$ are all distinct modulo~$3$, so $\widetilde b_{k - \ell} \,\widetilde b_{-\ell} \,\widetilde b \equiv \widetilde b_1 \, \widetilde b_{-1} \, \widetilde b \equiv e \pmod{\ZZ_2}$, but $b_{-j + k - \ell}$ is obviously nontrivial$\pmod{\ZZ_2}$. (Then, since the exponent of~$\widetilde a$ is~$0$, we must also have $i = j$.)
We conclude from \cref{WordMakesAut} that $\varphi(v s) = \varphi(v) \, \widetilde s$, for all $v \in \langle a, b \rangle = G$ and $s \in \{a^{\pm1}, b \}$, so $\varphi \in \Aut G$.

We may now assume $\langle a, s \rangle \neq G$, for all $s \in S$. Then, since $b \notin A_4$ (and $b$ is an element of order~$2$ in~$S$), we see that $b \in Z(G)$. Since $Z(G)$ has only one nontrivial element, this implies that $S = (S \cap A_4) \cup \{b\}$, and that $\widetilde b = b$ (since only $b$-edges make $4$-cycles with the edges of every other colour). Therefore
	$$\Cay(G;S) \iso \Cay(A_4; S \cap A_4) \times \Cay \bigl( \ZZ_2; \{b\} \bigr) ,$$
and $\varphi(b) = b$. Since $A_4$ is strongly CCA, it is now easy to see that $\varphi \in \Aut G$.
\end{proof}

\begin{eg} \label{63isCCA}
The following group of order $63$ is CCA:
	$$ G = \ZZ_7 \rtimes \ZZ_9 = \langle\, x, a \mid x^7 = a^9 = e, \ a^{-1} x a = x^2 \,\rangle .$$
\end{eg}

\begin{proof}
Let $\varphi$ be a colour-preserving automorphism of a connected Cayley graph $\Cay(G;S)$, such that $\varphi(e) = e$. We may assume $S$~is either $\{a^{\pm1}, x^{\pm1}\}$ or $\{a^{\pm1}, (ax)^{\pm1} \}$, after discarding redundant generators, applying an automorphism of~$G$, and replacing some elements by appropriate powers (cf.~the proof of \cref{AssumeSPrimePower}).

If $S = \{a^{\pm1}, x^{\pm1}\}$, then we may assume $\varphi(x) = x$, by composing with an automorphism of~$G$. Also, since $\varphi$ is colour-preserving, it must pass to a well-defined automorphism of the cycle $\Cay \bigl( G/\langle x \rangle; \{a^{\pm1}\} \bigr)$, so there exists $\epsilon \in \{\pm1\}$, such that $\varphi(ga) = \varphi(g)\,a^\epsilon$ for all $g \in G$. Then, since $(1,1)$ is the only pair $(\epsilon,\delta) \in \{\pm1\}^2$ that satisfies $a^{-\epsilon} x^\delta a^{\epsilon} = x^2$, we see that $\varphi( x^i a^j) = x^i a^j$ for all $i$ and~$j$, so $\varphi$ is the identity map, which is certainly affine.

Assume, now, that $S = \{a^{\pm1}, (ax)^{\pm1}\}$. Let $a_1 = a$ and $a_2 = xa$. For any $g \in G$ and $\epsilon \in \{\pm1\}$, if $\varphi(g\, a_1) = \varphi(g) \, a_1^\epsilon$, then, since $a_1^3 = a_2^3$ (and $\varphi$ is colour-preserving), we have $\varphi(g\, s^m) = \varphi(g) \, s^{\epsilon m}$, for all~$m$ and all $s \in S$. Since $S$ generates~$G$, this implies $\varphi(g s) = \varphi(g) \, s^\epsilon$ for all $g$ and all $s \in S$. So $\varphi$ is affine.
\end{proof}

\end{appendix}

\end{document}